\documentclass{amsart}
\usepackage[margin=1.3in]{geometry}
\usepackage{amsmath}
\usepackage{amsfonts}
\usepackage{amsthm}
\usepackage{comment}
\usepackage{url}
\usepackage{color}
\usepackage{enumitem}
\setlist[enumerate]{font=\normalfont}
\usepackage{mathtools}
\usepackage{amsrefs}
\usepackage{latexsym}
\usepackage{amssymb}
\usepackage{graphicx}        
\usepackage{float}
\usepackage{tikz-cd}
\usepackage{bm}
\usepackage[all]{xy}
\usepackage{ mathrsfs }
\usepackage[new]{old-arrows}
\tikzcdset{scale cd/.style={every label/.append style={scale=#1},
		cells={nodes={scale=#1}}}}

\usepackage[colorinlistoftodos]{todonotes}
\usepackage[colorlinks=true, allcolors=blue]{hyperref}
\usepackage{subcaption}
\usepackage{mathdots}
\usepackage{stmaryrd}

\theoremstyle{plain}
\newtheorem*{theorem*}{Theorem}
\newtheorem{theorem}{Theorem}[section]
\newtheorem{corollary}[theorem]{Corollary}
\newtheorem{lemma}[theorem]{Lemma}

\newtheorem{proposition}[theorem]{Proposition}

\numberwithin{equation}{section}

\theoremstyle{definition}
\newtheorem{definition}[theorem]{Definition}
\newtheorem{example}[theorem]{Example}

\newtheorem*{example*}{Example}

\theoremstyle{remark}
\newtheorem{remark}[theorem]{Remark}
\newtheorem*{remark*}{Remark}

\author{Juan Sebastian Numpaque-Roa}
\address{\flushleft Centro de Matemática da Universidade do Porto\\ Departamento de Matemática, Faculdade de Ciências da Universidade do Porto \\ Rua do Campo Alegre S/N, 4169-007 Porto, Portugal}
\email{\noindent js.numpaqueroa@gmail.com}

\begin{document}
	\title{Tensor Products of Quiver Bundles}
	\begin{abstract}
		In this work we introduce a notion of tensor product of (twisted) quiver representations with relations in the category of $\mathcal{O}_X$-modules. As a first application of our notion,  we see that tensor products of polystable quiver bundles are polystable and later we use this to both deduce a quiver version of the Segre embedding and to identify distinguished closed subschemes of $\text{GL}(n,\mathbb{C})$-character varieties of free abelian groups. 
	\end{abstract}
	\maketitle
	\tableofcontents
	\section{Introduction}
	Moduli spaces of quiver representations are ubiquitous objects in contemporary algebraic geometry. They provide geometric interpretations to several linear algebra and representation theoretic problems and are a natural playground to both test and infere results relevant for moduli theory (see, for instance, \cite{HoskinsSurvey}). A quiver is a finite directed graph and a quiver representation, in the classical sense, is obtained by a assigning to each vertex of the graph a finite-dimensional complex vector space and to each edge, a linear map between the previously assigned vector spaces. One can also consider representations of a quiver in other categories such as that of (quasi)-coherent sheaves over a complex projective variety or the more general category of $\mathcal{O}_X$-modules with $(X,\mathcal{O}_X)$ a ringed space. 
	\\
	\\ 
	A motivation for studying these more general quiver representations comes from the  study of Higgs bundles over a compact Riemann surface $X$. A Higgs bundle over $X$ is a pair $(\mathcal{E},\varphi:\mathcal{E}\to \mathcal{E}\otimes K_X)$ with $\mathcal{E}\to X$ a holomorphic vector bundle, $K_X$ the canonical bundle of $X$ and $\varphi$ a morphism of vector bundles. Moduli spaces for Higgs bundles of fixed topological types, which satisfy a certain stability condition, can be constructed  using both analytic and algebraic methods (see, for instance, \cite{Nitsure}). A successful strategy, as the works of Hitchin \cite{HitchinSeminal} and Gothen \cite{GothenBettiNumbers} show, to study the geometry of this moduli space has been to consider the natural $\mathbb{C}^*$-action on it given by 
	$
	t\cdot(\mathcal{E},\varphi)\mapsto (\mathcal{E},t\varphi).
	$ 
	It can be shown (see, for instance, \cite{GarciaRabosoRayan}) that the fixed point loci for this action correspond to moduli spaces of ``diagrams'', known, in this particular case, as holomorphic chains and, more generally, as quiver bundles, of the form:
	\begin{center}
		\begin{tikzcd}[scale cd=0.9]
			\mathcal{E}_m\arrow[r,"\varphi_{m}"]&\mathcal{E}_{m-1}\arrow[r,"\varphi_{m-1}"] &\cdots \arrow[r,"\varphi_{2}"]&\mathcal{E}_1.
		\end{tikzcd}
	\end{center}
	Here, each $\mathcal{E}_i\to X$ is a holomorphic vector bundle and $\varphi_i:\mathcal{E}_i\to \mathcal{E}_{i-1}$ is a morphism of holomorphic vector bundles. Many topological invariants, such as the cohomology, of the moduli space of Higgs bundles depend only on that of the fixed point loci hence the importance of the study of moduli spaces of quiver bundles. The careful reader may have already noticed that these are quiver representations in the category of holomorphic vector bundles over $X$.
	\\
	\\
	Moduli spaces for quiver bundles were constructed algebraically by Schmitt \cite{SchmittModuliQuiverBundles} and they depend on the choice of a stability parameter. Successful approaches to study the geometry of these moduli spaces include wall-crossing arguments involving the variation of the stability parameter (see, for instance, \cite{AlvarezConsulGarciaPradaSchmitt}).  On this regard, Gothen and Nozad \cite{GothenNozad} introduced a new quiver bundle whose stability simplifies and gives algebraic proofs to results which were previously only obtained by analytic methods. This is the starting point for this project: it turns out that Gothen and Nozad's quiver bundle is an instance of the notion of tensor product of quiver bundles that we introduce in this work (see Example \ref{extendedHomQuiverExample}). This notion generalizes, not only to quiver bundles but to the more general framework of quiver representations in the category of $\mathcal{O}_X$-modules,  that of Herschend \cite{HerschendTensorProducts} for classical quiver representations. It relies, basically, on the fact that quiver representations can be realized as modules over the path algebra of the quiver so one defines the tensor product of representations as that of the corresponding modules. Then one sees that the module obtained by tensorization corresponds to a representation with relations of a new quiver. 
	\\
	\\
	In Sections \ref{tensorProductOfModulesSection} and \ref{quiversWithRelationsSection}, after some preliminary definitions and results on quivers and their representations, we give the ingredients on which our notion is build. Theorem \ref{isomorphismTensorProductTwistedPathAlgebrasWithRelations} and Theorem \ref{equivalenceOfCatTwistedRepresentationsModulesWithRelations} are relevant, the former tells us that tensor products of quiver path algebras can be regarded as the path algebra of a new quiver modulo some ideal and the latter tells us how to go from modules over the path algebra to representations with relations of a given quiver. Then, in Section \ref{tensorProductsTwistedRepresentations} we put all together and discuss on our notion of tensor product of quiver representations. All the work until here is carried out for the more general twisted quiver representations on the category of $\mathcal{O}_X$-modules. In Section \ref{theUntwistedCase}, we discuss what happens when there is no twisting in our representations.  In Section \ref{surveySection} we survey both the algebraic and analytic constructions of the moduli space of classic quiver representations so as the gauge-theoretic quiver vortex equation on a compact Riemann surface. The main purpose of this is to lay the ground for the next section where applications will be presented. The first application concerns the polystability of tensor product of quiver bundles in Theorem \ref{polystabilityTensorProductQuiverBundles}, our approach will be analytic as we use the quiver vortex equations. As a corollary, we then recover a recent result from Das et.al. regarding the polystability of tensor products of classical quiver representations \cite{DasEtAlTensors}. After a brief detour, in which we discuss the effect that operations on quivers have on the stability of its representations (Section \ref{operationsOnQuiversAndStability}), we prove an analogue of the classic Segre embedding for classical quiver representations in Section \ref{segreEmbeddingForQuiverRepresentations}. In fact, we see that the former can be recovered from our quiver version (see Example \ref{classicalSegreEmbedding}).  The last application is related to character varieties of free abelian groups. Following the stream of ideas from Section \ref{segreEmbeddingForQuiverRepresentations}, we finally see in Section \ref{characterVarietiesSection} how tensorization of certain quiver representations can be used to obtain distinguished closed subschemes of the $\text{GL}(n)$-character variety of the free abelian group $\mathbb{Z}^r$.
	\newline
	\newline
	\textbf{Funding.} This work was supported by Fundação para a Ciência e a Tecnologia, I.P. [grant UI/BD/154369/2023]; and partially supported by Centro de Matemática da Universidade do Porto [grants UIDB/00144/2020, UIDP/00144/2020].
	\newline
	\newline
	\textbf{Acknowledgements.}
	I am deeply grateful with my adivsor, Prof. Peter B. Gothen, for introducing me to the topic, for stimulating conversations and for his help on more than one technical aspect of this work. Also, Veronica Calvo deserves a mention for her help on the proof of Lemma \ref{algebraicMorphismFromJordanQuivers}.
	\section{Representation theory of tensor quivers}\label{representationTheoryOfTensorQuivers}
	\subsection{Basic definitions}\label{basicDefinitions}
	We recall that a \emph{quiver} is a finite directed graph, that is, a tuple $Q=(Q_0,Q_1,h,t)$ where $Q_0$ is a finite set of vertices, $Q_1$ is a finite set of edges and $h,t:Q_1\to Q_0$ maps assigning to each edge its corresponding head and tail.  Let us present some examples of quivers that will play an important role in what follows:
	\noindent
	\begin{example}
		A quiver as in the following figure is known as a type $A_m$ quiver. 
		\vspace{-4.5mm}
		\begin{figure}[H]
			\begin{center}
				\begin{tikzcd}[scale cd=0.8, column sep =large, row sep =large]
					\bullet_m \arrow[r,"\alpha_m"]& \bullet_{m-1} \arrow[r,"\alpha_{m-1}"] &\bullet_{m-2} \arrow[r]&\cdots \arrow[r]& \bullet_{1}
				\end{tikzcd}
			\end{center}
			\caption{\label{typeAmQuiver}}
		\end{figure}

	\end{example}
	
	\begin{example}\label{kroneckerQuiverExample} The $n$-edge \emph{Kronecker quiver} is the quiver having two vertices, $\{1,2\}$, and $n$-edges, $\alpha_1,\ldots,\alpha_n$ such that $t\alpha_i=1$ and $h\alpha_i=2$ for all $i=1,\ldots,n$ as depicted in the figure below.
	\end{example}
	\vspace{-5.5mm}
	\begin{figure}[H]
		\begin{tikzcd}[scale cd=0.8]
			\bullet \arrow[to=1-3,bend left=25,"\alpha_1"]\arrow[to=1-3,bend right=25,"\alpha_n",']& \vdots & \bullet
		\end{tikzcd}
		\caption{\label{kroneckerQuiver}}
	\end{figure}
	\vspace{-5.5mm}
	\begin{example}\label{oppositeQuiverExample}
		Given a quiver $Q$, the \emph{opposite quiver} $Q^{op}$ is the quiver such that $Q_0^{op}=Q_0$ and $Q_1^{op}=\{\alpha^*\}_{\alpha\in Q_1}$ with $h\alpha^*=t\alpha$ and $t\alpha^*=h\alpha$. In other words, $Q^{op}$ is the quiver obtained from $Q$ by reversing its arrows. 
	\end{example}
	\begin{example}\label{tensorQuiverExample}
		Let $Q',Q''$ be two quivers. We define the \emph{tensor quiver} $Q'\otimes Q''$ by the following data: $(Q'\otimes Q'')_0=Q'_0\times Q''_0$, $(Q'\otimes Q'')_1=(Q'_1\times Q''_0)\sqcup (Q'_0\times Q''_1)$ and for $(\alpha,j)\in Q_1'\times Q_0'',(i,\beta)\in Q_0'\times Q_1''$ we set $h(\alpha,j)=(h\alpha,j),t(\alpha,j)=(t\alpha,j)$ and $h(i,\beta)=(i,h\beta),t(i,\beta)=(i,t\beta)$ respectively. Figure \ref{tensorProductA3Quivers} shows a concrete example of a tensor quiver. This is obtained by taking the ``tensor product'' of two quivers of type $A_3$. 
		\begin{figure}[H]
			\begin{center}
				\begin{tikzcd}[scale cd=0.8, column sep =large, row sep =large]
					& & \bullet_{(1,3)} \arrow[dr,,"(1{,}\,\alpha''_3)"]&&  \\
					&\bullet_{(2,3)}\arrow[ur,"(\alpha'_2{,} \, 3)"] \arrow[dr,"(2{,}\,\alpha''_3)"]& &\bullet_{(1,2)}   \arrow[dr,"(1{,}\,\alpha''_2)"]&   \\
					\bullet_{(3,3)} \arrow[ur,"(\alpha'_3{,} \, 3)"] \arrow[dr,"(3{,}\,\alpha''_3)",']&  & 	\bullet_{(2,2)} \arrow[ur,"(\alpha'_2{,} \, 2)",'] \arrow[dr,"(2{,}\,\alpha''_2)",']&  & 	\bullet_{(1,1)}\\
					&\bullet_{(3,2)}\arrow[ur,"(\alpha'_3{,} \, 2)"]  \arrow[dr,"(3{,}\,\alpha''_2)",']& &\bullet_{(2,1)} \arrow[ur,"(\alpha'_2{,} \, 1)",']&  \\
					& & \bullet_{(3,1)} \arrow[ur,"(\alpha'_3{,} \, 1)",'] &&
				\end{tikzcd}.
			\end{center}
			\caption{\label{tensorProductA3Quivers} Tensor product of two $A_3$ quivers}
		\end{figure}
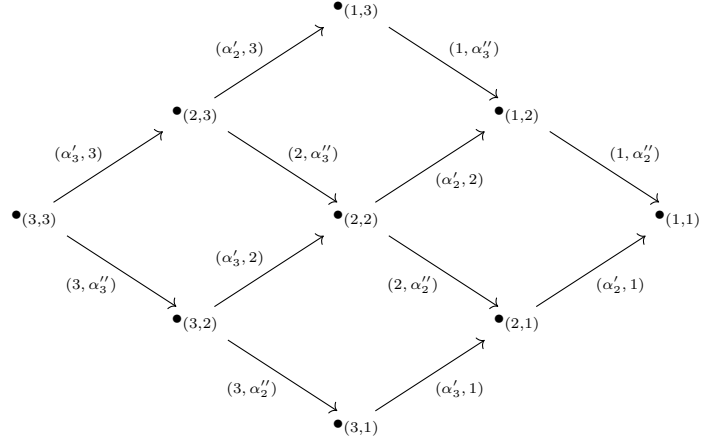
	\end{example}
	We fix a commutative ringed space $(X,\mathcal{O}_X)$ once and for all and we let $\mathscr{M}=\{\mathscr{M}_\alpha\}_{\alpha\in Q_1}$ be a family of $\mathcal{O}_X$-modules.
	\begin{definition}
		A \emph{$\mathscr{M}$-twisted representation} of $Q$ is given by a tuple $((\mathcal{E}_i)_{i\in Q_0},(\varphi_\alpha)_{\alpha\in Q_1})$ where $\mathcal{E}_i$ is an $\mathcal{O}_X$-module for all $i\in Q_0$ and $\varphi_{\alpha}:\mathscr{M}_\alpha\otimes_{\mathcal{O}_X}\mathcal{E}_{t\alpha} \to \mathcal{E}_{h\alpha}$ is a morphism of $\mathcal{O}_X$-modules for all $\alpha\in Q_1$.
	\end{definition}
	\begin{definition}
		Let $\mathcal{E},\mathcal{E}'$ be two $\mathscr{M}$-twisted representations of a quiver $Q$. A \emph{morphism} $f:\mathcal{E}\to \mathcal{E}'$ is given by a family $(f_i:\mathcal{E}_i\to \mathcal{E}'_i)_{i\in Q_0}$ with $f_i$ a morphism of $\mathcal{O}_X$-modules for which the diagram  
		\begin{equation}\label{twistedRepMorphismDiagram}
			\begin{tikzcd}[column sep =large, row sep =large]\
				\mathscr{M}_\alpha\otimes_{\mathcal{O}_X}\mathcal{E}_{t\alpha} \arrow[r,"\varphi_\alpha"] \arrow[d,'," \text{Id}\otimes f_{t\alpha}"] & \mathcal{E}_{h\alpha}\arrow[d,"f_{h\alpha}"] \\
				\mathscr{M}_\alpha\otimes_{\mathcal{O}_X} \mathcal{E}'_{t\alpha}\arrow[r,',"\varphi'_\alpha"] & \mathcal{E}'_{h\alpha}
			\end{tikzcd}
		\end{equation}
		commutes for all $\alpha\in Q_1$. 
	\end{definition}
	Thus, $\mathscr{M}$-twisted representations of $Q$ define a category which we denote $\text{Rep}(\mathscr{M}Q)$. 
	\begin{remark}\label{representationsInOtherCategories}
		We can restrict and work instead with the full subcategories of coherent sheaves over a compact Kähler manifold or the category of (quasi)-coherent sheaves over a scheme. All of our results remain valid on these. 
	\end{remark}
	One of the most prominent examples of twisted representations of a quiver is that of Higgs bundles:
	\begin{example}\label{exampleHiggsBundles}
		Suppose that $X$ is a compact Riemann surface. Recall, from the introduction, that a \emph{Higgs bundle} over $X$ is a pair $(\mathcal{E},\phi:\mathcal{E}\to \mathcal{E}\otimes K_X)$.  One can think of such an object as a $K_X^*$-twisted quiver representation, $K_X^*\to X$ being the dual of the canonical bundle, of the so-called \emph{Jordan quiver}, $Q_J$, in Figure \ref{JordanQuiver}.
		\vspace{-9mm}
		\begin{figure}[H]
			\begin{center}
				\begin{tikzcd}[scale cd=0.8, column sep =normal, row sep =normal]
					\bullet \arrow[loop,in=45,out=315,distance= 20mm]
				\end{tikzcd}
			\end{center}
			\vspace{-9mm}
			\caption{\label{JordanQuiver}The Jordan quiver $Q_J$}
		\end{figure}
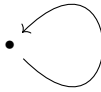
	\end{example}	
	An approach for studying representations of a quiver $Q$ has been to consider its so-called path algebra and studying instead the left modules over this algebra. To make this clear,  let us give some definitions:
	\begin{definition}\label{defPath}
		A \emph{non-trivial path} in $Q$ is a sequence $\alpha_1\alpha_2\cdots\alpha_k$ of edges such that $t\alpha_n=h\alpha_{n+1}$ for all $n=1,\ldots,k$. On the other hand, a \emph{trivial path} is the path that starts and terminates at $i$ for every vertex $i\in Q_0$. We denote this path by $e_i$.
	\end{definition}
	\begin{definition}
		The \emph{length} of a path $p$, which we denote by $|p|$, is given by the number of edges composing it.
	\end{definition}
	For instance,  those paths of length zero are the trivial paths $e_i$ for all $i\in Q_0$ and a path of length one would be any edge of the quiver $Q$. 
	
	\subsection{The path algebra}  We closely follow Álvarez-Cónsul and García-Prada \cite[Section 5.1]{AlvarezConsulGarciaPradaHKCorrespondenceQuiversVortices}. Fix a quiver $Q=(Q_0,Q_1,h,t)$ and consider the free $\mathcal{O}_X$-module
	$$
	\mathcal{B}=\bigoplus_{i\in Q_0}\mathcal{O}_X\cdot e_i
	$$
	with structure of commutative $\mathcal{O}_X$-algebra given by $e_i\cdot e_j=\delta_{ij}e_i$ for all $i,j\in Q_0$. Let $\{\mathscr{M}_\alpha\}_{\alpha\in Q_1}$ be a collection of $\mathcal{O}_X$-modules. Each of the $\mathscr{M}_\alpha$ can be given structure of $\mathcal{B}$-bimodule as follows: for all open subset $U\subseteq X$ and for every $m\in \mathscr{M}_\alpha(U)$, 
	$$
	e_i\cdot m=\begin{cases}
		m & \text{if } h\alpha=i\\
		0 & \text{otherwise} \\
	\end{cases}\ , \ m\cdot e_i=\begin{cases}
		m & \text{if }t\alpha=i \\ 
		0 & \text{otherwise}. 
	\end{cases}
	$$
	\begin{definition}
		Let $\mathscr{M}=\bigoplus_{\alpha\in Q_1}\mathscr{M}_\alpha$. The $\mathscr{M}$-\emph{twisted path algebra} of $Q$ is the tensor algebra of $\mathscr{M}$ over $\mathcal{B}$. We denote this algebra as $\mathcal{T}_\mathscr{M}\mathcal{A}_Q$. 
	\end{definition}
	\begin{remark}
		Since $Q$ has a finite number of vertices, 
		\begin{equation}\label{UnitPathAlgebra}
			1_{\mathcal{T}_\mathscr{M}\mathcal{A}_{Q}}=\oplus_{i\in Q_0}e_i
		\end{equation}
		is an unit for the path algebra $\mathcal{T}_{\mathscr{M}}\mathcal{A}_Q$.
	\end{remark}
	\begin{remark}
		Equivalently, following Gothen and King \cite{GothenKingHomologicalAlgebra}, one can define the twisted path algebra $\mathcal{T}_\mathscr{M}\mathcal{A}_Q$ as follows: for every non-trivial path $p=\alpha_1\alpha_2\cdots\alpha_k$, let $\mathscr{M}_p=\mathscr{M}_{\alpha_1}\otimes_{\mathcal{O}_X}\cdots\otimes_{\mathcal{O}_X}\mathscr{M}_{\alpha_k}$ and for every trivial path $p=e_i$ let $\mathscr{M}_i:=\mathscr{M}_p=\mathcal{O}_X$  so one can also define $\mathcal{T}_\mathscr{M}\mathcal{A}_Q=\bigoplus_{p \text{ path}}\mathscr{M}_p$ with multiplication given by:
		$$
		m_p\cdot m_q=\begin{cases}
			m_p\otimes m_q& \text{if } hq=tp, \\
			0 & \text{otherwise}
		\end{cases}
		$$
		for $p,q$ any two paths of the quiver $Q$ and $m_p\in \mathscr{M}_p(U),m_q\in \mathscr{M}_q(U)$ with $U\subseteq X$ an open set. 	Here, $hp$ and $tp$ denote the head and the tail of the path $p$ respectively or, in other words, the vertices where the path terminates and begins. 
	\end{remark}
In what follows, we will adopt either perspective on the twisted path algebra, depending on which best suits our purposes.
	\subsection{Modules over the path algebra}
	Now, we define modules over the twisted path algebra of a quiver $Q$ and we relate these to representations of $Q$ in the category of locally ringed spaces:
	\begin{definition}\label{defModuleOverPathAlgebra}
		A \emph{left $\mathcal{T}_\mathscr{M}\mathcal{A}_Q$-module} or simply a \emph{$\mathcal{T}_\mathscr{M}\mathcal{A}_Q$-module} is an $\mathcal{O}_X$-module  which is also a left $\mathcal{T}_\mathscr{M}\mathcal{A}_Q$-module. Equivalently, a $\mathcal{T}_\mathscr{M}\mathcal{A}_Q$-module is an $\mathcal{O}_X$-module $\mathcal{M}$ together with an $\mathcal{O}_X$-linear map
		$$
		\mu: \mathcal{T}_\mathscr{M}\mathcal{A}_Q\otimes_{\mathcal{O}_X}\mathcal{M}  \longrightarrow  \mathcal{M}
		$$
		for which the usual axioms of left modules over an algebra hold. 
		A \emph{morphism of $\mathcal{T}_\mathscr{M}\mathcal{A}_Q$-modules} is an $\mathcal{O}_X$-linear map $f:\mathcal{M}\to \mathcal{N}$ such that the diagram 
		\begin{equation}\label{modMorphismDiagram}
			\begin{tikzcd}[column sep =large, row sep =large]\
				\mathcal{T}_\mathscr{M}\mathcal{A}_Q\otimes_{\mathcal{O}_X} \mathcal{M}\arrow[r,"\mu"] \arrow[d,',"\text{Id}\otimes f"] & \mathcal{M}\arrow[d,"f"] \\
				\mathcal{T}_\mathscr{M}\mathcal{A}_Q\otimes_{\mathcal{O}_X} \mathcal{N} \arrow[r,',"\mu"] & \mathcal{N}
			\end{tikzcd}
		\end{equation}
		commutes.
	\end{definition}
	
	The following proposition generalizes the well known fact for classical quiver representations relating modules and quiver representations  (see, for instance, \cite[Theorem 1.7]{libroKirillov}). Recently, Bartocci, Bruzzo and Rava \cite{BartocciHomologyTwistedQuiverBundles} gave a proof that works in the more general framework of twisted representations in the category of $\mathcal{O}_X$-modules. They ask, however, for the local projectivity of the twisting $\mathcal{O}_X$-modules $\mathscr{M}_\alpha$. After reviewing their proof we think that this hypothesis may be avoided and that this can be made to work without imposing any conditions on the twisting modules:
	\begin{proposition}\label{equivCategoriesTwistedModulesReps}
		The category of left $\mathcal{T}_\mathscr{M}\mathcal{A}_Q$-modules, {\normalfont $\mathcal{T}_\mathscr{M}\mathcal{A}_Q$-mod},  is equivalent to the category {\normalfont $\text{Rep}(\mathscr{M} Q)$}. 
	\end{proposition}
	\begin{proof}
		Let $\mathcal{E}=((\mathcal{E}_i)_{i\in Q_0},(\varphi_\alpha)_{\alpha\in Q_1})$ be a $\mathscr{M}$-twisted representation. We endow the $\mathcal{O}_X$-module $G(\mathcal{E})=\bigoplus_{i\in Q_0}\mathcal{E}_i$ with structure of $\mathcal{T}_\mathscr{M}\mathcal{A}_Q$-module. This, following Definition \ref{defModuleOverPathAlgebra}, is the same as giving an $\mathcal{O}_X$-linear map $\mu:\mathcal{T}_\mathscr{M}\mathcal{A}_Q\otimes_{\mathcal{O}_{X}}G(\mathcal{E})\to G(\mathcal{E})$ or, equivalently by the Tensor-Hom adjunction \cite[p.180]{GoertzAlgebraicGeometry}, an $\mathcal{O}_X$-linear map $\mathcal{T}_\mathscr{M}\mathcal{A}_Q\to \mathcal{E}nd_{\mathcal{O}_X}(G(\mathcal{E}))$. 
		\\
		\\
		The $\mathcal{O}_X$-module $G(\mathcal{E})$ can be given structure of left $\mathcal{B}$-module by taking $e_i\cdot s_j=\delta_{ij}s_j$ for $s_j$ a section of $\mathcal{E}_j$ and $i\in Q_0$. By the universal property of the tensor algebra $\mathcal{T}_\mathscr{M}\mathcal{A}_Q$, to give a map $\mathcal{T}_\mathscr{M}\mathcal{A}_Q\to \mathcal{E}nd_\mathcal{B}(G(\mathcal{E}))$ is enough to give a $\mathcal{B}$-linear map $\mathscr{M}\to\mathcal{E}nd_\mathcal{B}(G(\mathcal{E}))$. For a fixed $m_\alpha\in \mathscr{M}_\alpha(U)$ with $U\subseteq X$ open, let $\psi_{m_\alpha}:\mathcal{E}_{t\alpha}|_U\to \mathcal{E}_{h\alpha}|_U$ be the sheaf morphism given by the  composition:
		$$
		\begin{matrix}
			\mathcal{E}_{t\alpha}(V) &\longrightarrow & \mathscr{M}_\alpha(V)\otimes_{\mathcal{O}_{X}(V)}\mathcal{E}_{t\alpha}(V) & \overset{\widetilde{\varphi}_\alpha}{\longrightarrow}& \mathcal{E}_{h\alpha}(V)\\ 
			s_{t\alpha} &\longmapsto& m_\alpha \otimes s_{t\alpha} & \longmapsto& \widetilde{\varphi}_\alpha(m_\alpha\otimes s_{t\alpha}) 
		\end{matrix}
		$$
		where $V\subseteq U$ is an open set and $\widetilde{\varphi}_\alpha$ is the presheaf morphism corresponding to the morphism $\varphi_\alpha:\mathscr{M}_\alpha\otimes_{\mathcal{O}_X}\mathcal{E}_{t\alpha}\to \mathcal{E}_{h\alpha}$. With this in mind we define for all open $U\subseteq X$
		$$
		\begin{matrix}
			\mathscr{M}(U)&\longrightarrow & \text{End}_{B|_U}(G(\mathcal{E})|_U) \\
			m_\alpha & \longmapsto & \psi_{m_\alpha}
		\end{matrix}.
		$$
		Composing the obtained map on the tensor algebra $\mathcal{T}_\mathscr{M}\mathcal{A}_Q\to \mathcal{E}nd_\mathcal{B}(G(\mathcal{E}))$ with the inclusion $\mathcal{E}nd_\mathcal{B}(G(\mathcal{E}))\to \mathcal{E}nd_{\mathcal{O}_X}(G(\mathcal{E}))$ we get the desired $\mathcal{T}_\mathscr{M}\mathcal{A}_Q$-module structure on $G(\mathcal{E})$.  
		\\
		\\
		Now, let $f=(f_i)_{i\in Q_0}:\mathcal{E}\to \mathcal{E}'$ be a morphism of $\mathscr{M}$-twisted representations. We claim that $G(f)=\oplus_{i\in Q_0}f_i:G(\mathcal{E})\to G(\mathcal{E}')$ is a morphism of $\mathcal{T}_\mathscr{M}\mathcal{A}_Q$-modules so we need to check the commutativity of Diagram \ref{modMorphismDiagram}. This essentially follows from the commutativity of Diagram \ref{twistedRepMorphismDiagram}. 
		\\
		\\
		Conversely, let $\mathcal{E}$ be a left $\mathcal{T}_\mathscr{M}\mathcal{A}_Q$-module with structure morphism $\mu:\mathcal{T}_\mathscr{M}\mathcal{A}_Q\otimes_{\mathcal{O}_{X}}\mathcal{E}\to \mathcal{E}$. Let $\mathcal{E}_i:=\mu(\mathscr{M}_i\otimes_{\mathcal{O}_X} \mathcal{E})\subseteq\mathcal{E}$. Because of the condition $e_ie_j=\delta_{ij}$ for $i,j\in Q_0$ we have that $\mathcal{E}\cong \bigoplus_{i\in Q_0}\mathcal{E}_i$. For all $\alpha\in Q_1$, let $\varphi_{\alpha}:\mathcal{E}_{t\alpha}\to \mathcal{E}_{h\alpha}$ be the morphism of $\mathcal{O}_X$-modules associated to the morphism of presheaves of $\mathcal{O}_X$-modules given, for all $U\subseteq X$ open subset, by 
		$$
		\begin{matrix}
			\mathscr{M}_\alpha(U)\otimes_{\mathcal{O}_{X}(U)}\mathcal{E}_{t\alpha}(U)&\longrightarrow& \mathcal{E}_{h\alpha}(U)\\
			m_\alpha\otimes s_{t\alpha}&\longmapsto & \widetilde{\mu}(m_\alpha\otimes s_{t\alpha})
		\end{matrix}.
		$$
		Here, $\widetilde{\mu}$ is the presheaf morphism corresponding to the structure morphism $\mu$. The data $F(\mathcal{E})=((\mathcal{E}_i)_{i\in Q_0},(\varphi_{\alpha})_{\alpha\in Q_1})$ is that of a $\mathscr{M}$-twisted representation of $Q$. Finally, let $f:\mathcal{E}\to \mathcal{E}'$ be a morphism of $\mathcal{T}_\mathscr{M}\mathcal{A}_Q$-modules and let $F(f)=(f|_{\mathcal{E}_i}:\mathcal{E}_i\to \mathcal{E}'_i)_{i\in Q_0}$. The latter is a morphism of $\mathscr{M}$-twisted representations mainly because the commutativity of Diagram \ref{modMorphismDiagram}.
	\end{proof}
	
	\subsection{Tensor products of modules}\label{tensorProductOfModulesSection}
	Towards our notion of tensor product of quiver representations, we will now study tensor products of modules over the path algebra. Later on, we will see that these correspond to representations of the tensor quiver introduced in Example \ref{tensorQuiverExample} with some additional data. 
	\\
	\\
	Recall, from Definition \ref{defModuleOverPathAlgebra}, that given an $\mathcal{O}_X$-module $\mathcal{F}$, a $\mathcal{T}_\mathscr{M}\mathcal{A}_Q$-module structure on $\mathcal{F}$ is given by an $\mathcal{O}_X$-linear map 
	$$
	\mu_Q:\mathcal{T}_\mathscr{M}\mathcal{A}_Q\otimes_{\mathcal{O}_X}\mathcal{F}\to \mathcal{F} .
	$$
	With this in mind, let $Q'$, $Q''$ be quivers and $\mathscr{M}'=\{\mathscr{M}'_\alpha\}_{\alpha\in Q_1'}$, $\mathscr{M}''=\{\mathscr{M}''_\beta\}_{\beta\in Q_1''}$ be collections of $\mathcal{O}_X$-modules. Note that the tensor product $\mathcal{M}\otimes_{\mathcal{O}_X}\mathcal{N}$, with $\mathcal{M}$ a $\mathcal{T}_{\mathscr{M}'}\mathcal{A}_{Q'}$-module and $\mathcal{N}$ a $\mathcal{T}_{\mathscr{M}''}\mathcal{A}_{Q''}$-module, carries a canonical left $\mathcal{T}_{\mathscr{M}'}\mathcal{A}_{Q'}\otimes_{\mathcal{O}_X}\mathcal{T}_{\mathscr{M}''}\mathcal{A}_{Q''}$-module structure given by
	\begin{equation}\label{tensorProductOfModulesOverPathAlgebras}
		(\mathcal{T}_{\mathscr{M}'}\mathcal{A}_{Q'}\otimes_{\mathcal{O}_X}\mathcal{T}_{\mathscr{M}''}\mathcal{A}_{Q''})\otimes _{\mathcal{O}_X}(\mathcal{M}\otimes_{\mathcal{O}_X}\mathcal{N})\overset{\mu_{Q'}\otimes\mu_{Q''}}{\longrightarrow} \mathcal{M}\otimes_{\mathcal{O}_X}\mathcal{N}.
	\end{equation}
	On the other hand, the tensor product quiver $Q'\otimes Q''$ can be given canonically a $\mathscr{M}$-twisted path algebra $\mathcal{T}_\mathscr{M}\mathcal{A}_{Q'\otimes Q''}$ for
	$$
	\mathscr{M}=\bigoplus_{(\alpha,j)}\mathscr{M}_{(\alpha,j)}\oplus \bigoplus_{(i,\beta)}\mathscr{M}_{(i,\beta)},
	$$
	$\mathscr{M}_{(\alpha,j)}=\mathscr{M}_\alpha'$ and $\mathscr{M}_{(i,\beta)}=\mathscr{M}_\beta''$ for all $(\alpha,j),\ (i,\beta)\in (Q'\otimes Q'')_1$. 	Now we see that the left $\mathcal{T}_{\mathscr{M}'}\mathcal{A}_{Q'}\otimes_{\mathcal{O}_X}\mathcal{T}_{\mathscr{M}''}\mathcal{A}_{Q''}$-module $\mathcal{M}\otimes_{\mathcal{O}_{X}}\mathcal{N}$ can be seen as left module over the path algebra $\mathcal{T}_\mathscr{M}\mathcal{A}_{Q'\otimes Q''}$ modulo some ideal. This is inspired by Herschend's result for classical quiver representations \cite[Proposition 3]{HerschendTensorProducts}.
	\begin{proposition}\label{isomorphismTensorProductTwistedPathAlgebras}
		There is an isomorphism of twisted path algebras 
		$$
		\mathcal{T}_\mathscr{M}\mathcal{A}_{Q'\otimes Q''}/\mathcal{I}\cong \mathcal{T}_{\mathscr{M}'}\mathcal{A}_{Q'}\otimes_{\mathcal{O}_X}\mathcal{T}_{\mathscr{M}''}\mathcal{A}_{Q''} 
		$$
		where $\mathcal{I}$ is the two-sided ideal subsheaf associated to the ideal presheaf $\widetilde{\mathcal{I}}$ generated, for all open subset $U\subseteq X$, by all the differences 
		\begin{equation}\label{differencesIdeal}
			m_{(h\alpha,\beta)}\otimes m_{(\alpha,t\beta)}-m_{(\alpha,h\beta)}\otimes m_{(t\alpha,\beta)}, 
		\end{equation}
		for which $ m_{(\alpha,t\beta)}=m_{(\alpha,h\beta)}\in \mathscr{M}_{(\alpha,t\beta)}(U)=\mathscr{M}_{(\alpha,h\beta)}(U)=\mathscr{M}_\alpha'(U)$ and $ m_{(t\alpha,\beta)}=m_{(h\alpha,\beta)}\in \mathscr{M}_{(t\alpha,\beta)}(U)=\mathscr{M}_{(h\alpha,\beta)}(U)=\mathscr{M}_\beta''(U)$.
	\end{proposition}
	\begin{proof} For any given quiver $Q$, we will denote the presheaf of $\mathcal{O}_X$-algebras, $U\mapsto \bigoplus_{p \text{ path}}\mathscr{M}_p(U)$, as  $\widetilde{\mathcal{T}_\mathscr{M}\mathcal{A}_{Q}}$. Consider the morphism $\varphi:\widetilde{\mathcal{T}_\mathscr{M}\mathcal{A}_{Q'\otimes Q''}}\to \widetilde{\mathcal{T}_{\mathscr{M}'}\mathcal{A}_{Q'}}\otimes_{\mathcal{O}_X}\widetilde{\mathcal{T}_{\mathscr{M}''}\mathcal{A}_{Q''}}$ of presheaves of $\mathcal{O}_X$-algebras defined, on generators, by 
		$$
		\begin{matrix}
			1_{\mathscr{M}_{(i,j)}}&\longmapsto &  1_{\mathscr{M}_{i}'}\otimes 1_{\mathscr{M}_{j}''},\\  
			m_{(\alpha,j)} & \longmapsto & m_{(\alpha,j)}\otimes 1_{\mathscr{M}_{j}''},\\
			m_{(i,\beta)} & \longmapsto & 1_{\mathscr{M}_{i}'}\otimes m_{(i,\beta)}.
		\end{matrix}
		$$
		Note that  all the differences described in Equation \ref{differencesIdeal} are contained in the kernel of $\varphi$ so this map factors through the quotient by $\widetilde{\mathcal{I}}$ giving a map $\widetilde{\varphi}:\widetilde{\mathcal{T}_\mathscr{M}\mathcal{A}_{Q'\otimes Q''}}/\widetilde{\mathcal{I}}\to \widetilde{\mathcal{T}_{\mathscr{M}'}\mathcal{A}_{Q'}}\otimes_{\mathcal{O}_X}\widetilde{\mathcal{T}_{\mathscr{M}''}\mathcal{A}_{Q''}}$. We now show that $\widetilde{\varphi}$ is an isomorphism of $\mathcal{O}_X$-algebras and to this purpose we will give an inverse. Let $\psi: \widetilde{\mathcal{T}_{\mathscr{M}'}\mathcal{A}_{Q'}}\otimes_{\mathcal{O}_{X}}\widetilde{\mathcal{T}_{\mathscr{M}''}\mathcal{A}_{Q''}} \rightarrow \widetilde{\mathcal{T}_\mathscr{M}\mathcal{A}_{Q'\otimes Q''}}/\widetilde{\mathcal{I}}$ be the morphism of presheaves of $\mathcal{O}_X$-algebras given by 
		$$
		\begin{matrix}
			\psi_U: & \widetilde{\mathcal{T}_{\mathscr{M}'}\mathcal{A}_{Q'}}(U)\otimes_{\mathcal{O}_{X}(U)}\widetilde{\mathcal{T}_{\mathscr{M}''}\mathcal{A}_{Q''}}(U) & \longrightarrow & \widetilde{\mathcal{T}_\mathscr{M}\mathcal{A}_{Q'\otimes Q''}}(U)/\widetilde{\mathcal{I}}(U)\\
			& m_p\otimes m_q & \longmapsto & m_{(h\alpha_1,q)}\otimes m_{(p,t\beta_m)}\text{ mod }\widetilde{\mathcal{I}}(U).
		\end{matrix}
		$$
		Here $U\subseteq X$ is an open set and $p=\alpha_1\cdots\alpha_n$, $q=\beta_1\cdots\beta_m$ are paths in $Q'$ and $Q''$ respectively. Moreover, on the left-most side, $m_p=m_{\alpha_1}\otimes\cdots\otimes m_{\alpha_n}$, $m_{\alpha_k}\in \mathscr{M}_{\alpha_k}(U)$, and $m_q=m_{\beta_1}\otimes\cdots \otimes m_{\beta_m}$, $m_{\beta_k}\in\mathscr{M}_{\beta_k}(U)$, and on the right-most side $m_{(h\alpha_1,q)}=m_q$ and $m_{(p,t\beta_m)}=m_p$ now both understood as elements in $\mathscr{M}_{(h\alpha_1,\beta_1)}(U)\otimes\cdots \otimes \mathscr{M}_{(h\alpha_1,\beta_m)}(U)$ and $\mathscr{M}_{(\alpha_1,t\beta_m)}(U)\otimes\cdots\otimes \mathscr{M}_{(\alpha_n,t\beta_m)}(U)$ respectively.
		\\
		\\
		Following the same convention as before, for any open set $U\subseteq X$, we let $S(U)$ be the set of tensors of the form $m_{(hp,q)}\otimes m_{(p,tq)}\in \mathscr{M}_{(h\alpha,q)(p,t\beta)}(U)$ for any paths $p$, $q$ of $Q'$ and $Q''$ respectively. Let 
		$$
		r=(i_{n+1},\beta_n)(\alpha_n,j_n)\ldots(\alpha_2,j_2)(i_2,\beta_1)(\alpha_1,j_1)
		$$
		be any path of the tensor quiver $Q'\otimes Q''$ and let 
		$$
		m_r=m_{(i_{n+1},\beta_n)}\otimes m_{(\alpha_n,j_n)}\otimes \cdots\otimes m_{(\alpha_2,j_2)}\otimes m_{(i_2,\beta_1)}\otimes m_{(\alpha_1,j_1)}\in \mathscr{M}_r(U).
		$$ 
		We claim that this tensor is equivalent to some in $S(U)$ modulo $\widetilde{\mathcal{I}}(U)$.
		Successively replacing all the products, appearing in $m_r$, of the form $m_{(\alpha_k,j_k)}\otimes m_{(i_k,\beta_{k-1})}$ with $m_{(i_{k+1},\beta_{k-1})}\otimes m_{(\alpha_k,j_{k-1})}$ (here $m_{(\alpha_k,j_k)}=m_{(\alpha_k,j_{k-1})}$ and $m_{(i_k,\beta_{k-1})}=m_{(i_{k+1},\beta_{k-1})}$) we obtain a tensor of the form 
		$$
		m_{(i_{n+1},\beta_n)}\otimes\cdots\otimes m_{(i_{n+1},\beta_1)}\otimes m_{(\alpha_n,j_1)}\otimes \cdots\otimes m_{(\alpha_2,j_1)}\otimes m_{(\alpha_1,j_1)}
		$$ 
		and, moreover, by doing so we stay in the same class modulo $\widetilde{\mathcal{I}}(U)$.
		Hence the set $S(U) \text{ mod }\widetilde{\mathcal{I}}(U)$ spans $\widetilde{\mathcal{T}_\mathscr{M}\mathcal{A}_{Q'\otimes Q''}}(U)/\widetilde{\mathcal{I}}(U)$. Now it is easy to see that at the level of generators the maps $\widetilde{\varphi}$ and $\psi$ are inverse to one another and therefore $\widetilde{\varphi}$ is an isomorphism of presheaves of $\mathcal{O}_X$-algebras as claimed.
		\\
		\\
		Finally, we note that sheafification commutes with quotients which boils down to the fact that sheafification is an exact functor. Thus, the associated sheaf to the presheaf of $\mathcal{O}_X$-algebras $\widetilde{\mathcal{T}_\mathscr{M}\mathcal{A}_{Q'\otimes Q''}}/\widetilde{\mathcal{I}}$ is isomorphic to $\mathcal{T}_\mathscr{M}\mathcal{A}_{Q'\otimes Q''}/\mathcal{I}$. Also, we note that sheafification commutes with tensor products \cite[Lemma 17.16.2]{stacksProject}, hence the associated sheaf to the presheaf of $\mathcal{O}_X$-algebras $\widetilde{\mathcal{T}_{\mathscr{M}'}\mathcal{A}}_{Q'}\otimes_{\mathcal{O}_X}\widetilde{\mathcal{T}_{\mathscr{M}''}\mathcal{A}}_{Q''}$ is isomorphic to $\mathcal{T}_{\mathscr{M}'}\mathcal{A}_{Q'}\otimes_{\mathcal{O}_X}\mathcal{T}_{\mathscr{M}''}\mathcal{A}_{Q''}$. The statement of the proposition then follows as isomorphic presheaves have isomorphic associated sheaves. 
	\end{proof}
	
	The ideal $\mathcal{I}$ of the path algebra $\mathcal{T}_\mathscr{M}\mathcal{A}_{Q'\otimes Q''}$ of the previous proposition turns out to be the ideal generated by a set of relations on the tensor quiver $Q'\otimes Q''$. We will see that modules over the path algebra modulo such an ideal correspond to representations of the associated quiver for which the morphisms at each one of the edges of the quiver satisfy particular conditions
	\subsection{Quivers with relations}\label{quiversWithRelationsSection}
	\begin{definition}\label{definitionRelation}
		A \emph{relation} on a quiver $Q$ is a formal sum of the form
		$$
		\sum_{\substack{p\text{ path, }|p|\geq 2\\ tp=i,\ hp=j}}c_p\cdot p, \ \ c_p\in \mathcal{O}_X(X)
		$$
		where all but finitely many terms are zero.
	\end{definition}
	A set of relations $\mathcal{R}$ on a quiver defines a two-sided ideal of the twisted path algebra $\mathcal{T}_\mathscr{M}\mathcal{A}_Q$ which is generated by 
	$$
	\mathcal{R}(U):=\{\sum c_p\cdot m_p|\sum c_p\cdot p\in \mathcal{R},m_p\in \mathscr{M}_p(U) \}
	$$
	for all open $U\subseteq X$. In the rest of this paper we will be dealing with two-sided ideals of the twisted path algebra generated by a subset of $\mathcal{R}(U)$ for all open subset $U\subseteq X$ (see \cite[Lemma 17.4.4]{stacksProject}) and we will refer to these as \emph{ideals generated by the set of relations $\mathcal{R}$}. The prototypical example of such an ideal, which the reader should have in mind, is that of Proposition \ref{isomorphismTensorProductTwistedPathAlgebras}. Also, notice that this construction certainly applies to the presheaf of $\mathcal{O}_X$-algebras $\widetilde{\mathcal{T}_\mathscr{M}\mathcal{A}_{Q}}$.
	\\
	
	With the notion of relation just introduced in mind, before moving forward, we show a more general result in the spirit of Proposition \ref{isomorphismTensorProductTwistedPathAlgebras}. Under the hypotheses and notation of this proposition we have:
	\begin{theorem}\label{isomorphismTensorProductTwistedPathAlgebrasWithRelations}
		Let $\widetilde{\mathcal{J}}'$ and $\widetilde{\mathcal{J}}''$ be two-sided ideals, of $\widetilde{\mathcal{T}_{\mathscr{M}'}\mathcal{A}}_{Q'}$ and $\widetilde{\mathcal{T}_{\mathscr{M}''}\mathcal{A}}_{Q''}$ respectively, generated by sets of relations $\mathcal{R}'$ and $\mathcal{R}''$.  Let $\widetilde{\mathcal{I}}'$ and $\widetilde{\mathcal{I}}''$ be the two-sided ideals of the twisted path algebra $\widetilde{\mathcal{T}_\mathscr{M}\mathcal{A}_{Q'\otimes Q''}}$ generated by 
		$
		\sum_{j\in Q_0''}m_{(r',j)} \text{ and } \sum_{i\in Q_0'}m_{(i,r'')}
		$
		respectively. Here, for $m_{r'}=\sum m_p$ a generator of the ideal $\widetilde{\mathcal{J}}'$, $m_{(r',j)}=m_{r'}$ but we think now of each summand $m_p$ as an element in $\mathscr{M}_{(p,j)}$ rather than $\mathscr{M}_p$. Likewise for the terms $m_{(i,r'')}$. Then, there is an isomorphism 
		$$
		\mathcal{T}_\mathscr{M}\mathcal{A}_{Q'\otimes Q''}/\mathcal{I}+\mathcal{I}'+\mathcal{I}''\cong \mathcal{T}_{\mathscr{M}'}\mathcal{A}_{Q'}/\mathcal{J}'\otimes_{\mathcal{O}_X}\mathcal{T}_{\mathscr{M}''}\mathcal{A}_{Q''}/\mathcal{J}''
		$$
		where $\mathcal{I}'(\mathcal{I}'')$ is the associated two-sided ideal sheaf of $\mathcal{T}_\mathscr{M}\mathcal{A}_{Q'\otimes Q''}$ obtained from the sheafification of $\widetilde{\mathcal{I}}'(\widetilde{\mathcal{I}}'')$. 
	\end{theorem}
	\begin{proof}
		It is easy to see that $\widetilde{\mathcal{I}}+\widetilde{\mathcal{I}}'+\widetilde{\mathcal{I}}''\subseteq\ker(\text{pr}\circ\varphi)$ for $\text{pr}:\widetilde{\mathcal{T}_{\mathscr{M}'}\mathcal{A}_{Q'}}\otimes_{\mathcal{O}_X}\widetilde{\mathcal{T}_{\mathscr{M}''}\mathcal{A}_{Q''}}\to \widetilde{\mathcal{T}_{\mathscr{M}'}\mathcal{A}_{Q'}}/\widetilde{\mathcal{J}}'\otimes_{\mathcal{O}_X}\widetilde{\mathcal{T}_{\mathscr{M}''}\mathcal{A}_{Q''}}/\widetilde{\mathcal{J}}''$ the canonical projection. Thus we have a map 
		$$
		\widetilde{\text{pr}\circ\varphi}: 	\widetilde{\mathcal{T}_\mathscr{M}\mathcal{A}_{Q'\otimes Q''}}/\widetilde{\mathcal{I}}+\widetilde{\mathcal{I}}'+\widetilde{\mathcal{I}}''\longrightarrow \widetilde{\mathcal{T}_{\mathscr{M}'}\mathcal{A}_{Q'}}/\widetilde{\mathcal{J}}'\otimes_{\mathcal{O}_X}\widetilde{\mathcal{T}_{\mathscr{\mathscr{M}}''}\mathcal{A}_{Q''}}/\widetilde{\mathcal{J}}''.
		$$
		To show that this is an isomorphism we follow the same steps as in Proposition \ref{isomorphismTensorProductTwistedPathAlgebras}. One sees that $\psi$ factors through $\widetilde{\mathcal{T}_{\mathscr{M}'}\mathcal{A}_{Q'}}/\widetilde{\mathcal{J}}'\otimes_{\mathcal{O}_X}\widetilde{\mathcal{T}_{\mathscr{M}''}\mathcal{A}_{Q''}}/\widetilde{\mathcal{J}}''$  so we have a map
		$$\widetilde{\psi}:\widetilde{\mathcal{T}_{\mathscr{M}'}\mathcal{A}_{Q'}}/\widetilde{\mathcal{J}}'\otimes_{\mathcal{O}_X}\widetilde{\mathcal{T}_{\mathscr{M}''}\mathcal{A}_{Q''}}/\widetilde{\mathcal{J}}''\to\widetilde{\mathcal{T}_{\mathscr{M}}\mathcal{A}_{Q'\otimes Q''}}/\widetilde{\mathcal{I}}+\widetilde{\mathcal{I}}'+\widetilde{\mathcal{I}}''$$ which is an inverse for $\widetilde{\text{pr}\circ\varphi}$. This is checked, as in Proposition \ref{isomorphismTensorProductTwistedPathAlgebras}, at the level of generators and by noting that for all $U\subseteq X$, the set $S(U)\text{ mod }(\widetilde{\mathcal{I}}+\widetilde{\mathcal{I}}'+\widetilde{\mathcal{I}}'')(U)$ spans $\widetilde{\mathcal{T}_\mathscr{M}\mathcal{A}_{Q'\otimes Q''}}(U)/(\widetilde{\mathcal{I}}+\widetilde{\mathcal{I}}'+\widetilde{\mathcal{I}}'')(U)$. Thus we have an isomorphism of presheaves of $\mathcal{O}_X$-algebras 
		$$
		\widetilde{\mathcal{T}_{\mathscr{M}}\mathcal{A}_{Q'\otimes Q''}}/\widetilde{\mathcal{I}}+\widetilde{\mathcal{I}}'+\widetilde{\mathcal{I}}''\cong \widetilde{\mathcal{T}_{\mathscr{M}'}\mathcal{A}}_{Q'}/\widetilde{\mathcal{J}}'\otimes_{\mathcal{O}_X}\widetilde{\mathcal{T}_{\mathscr{M}''}\mathcal{A}}_{Q''}/\widetilde{\mathcal{J}}''.
		$$
		Our theorem then follows by the same arguments to those in Proposition \ref{isomorphismTensorProductTwistedPathAlgebras} regarding the commutativity of the sheafification functor with quotients and tensor products. The only thing that needs to be checked though is that the sheafification of  the presheaf $\widetilde{\mathcal{I}}+\widetilde{\mathcal{I}}'+\widetilde{\mathcal{I}}''$ is indeed $\mathcal{I}+\mathcal{I}'+\mathcal{I}''$. This follows from both the fact that $\widetilde{\mathcal{I}}+\widetilde{\mathcal{I}}'+\widetilde{\mathcal{I}}''$ is the image of $\widetilde{\mathcal{I}}\times\widetilde{\mathcal{I}}'\times\widetilde{\mathcal{I}}''$ under the sum morphism and the fact that the sheafification functor is exact.
	\end{proof}

	\begin{definition}\label{defTwistedRepresentationWithRelations}
		Let $\mathcal{R}$ be a set of relations of the quiver $Q$ and $\mathcal{I}_\mathcal{R}$ an ideal of $\mathcal{T}_\mathscr{M}\mathcal{A}_Q$ generated by $\mathcal{R}$, then a $\mathscr{M}$-twisted representation of $Q$, $\mathcal{E}=((\mathcal{E}_i)_{i\in Q_0},(\varphi_\alpha)_{\alpha\in Q_1})$, is said to \emph{satisfy the relations in} $\mathcal{\mathcal{I}_\mathcal{R}}$ if for all $r=\sum m_p$ generator of $ \mathcal{I}_\mathcal{R}(U)$ we have that 
		$$
		\sum \psi_{m_p}:\mathcal{E}_{tp}|_U\to \mathcal{E}_{hp}|_U\equiv 0.
		$$
		Here, $\psi_{m_p}:\mathcal{E}_{tp}|_U\to \mathcal{E}_{hp}|_U$ is the morphism given by the image of $m_p\in\mathscr{M}_p(U)$ under the map $\mathcal{T}_\mathscr{M}\mathcal{A}_Q\to \mathcal{E}nd_{\mathcal{O}_X}(G(\mathcal{E}))$ in Proposition \ref{equivCategoriesTwistedModulesReps}. We denote $\text{Rep}(\mathscr{M}Q,\mathcal{I}_\mathcal{R})$ the full subcategory of $\text{Rep}(\mathscr{M}Q)$ whose objects are the representations of $Q$ satisfying the relations in $\mathcal{I}_\mathcal{R}$.
	\end{definition}
	
	\begin{proposition}\label{fullEmbeddingPathAlgebras}
		Let $\mathcal{I}$ be a two-sided ideal of the path algebra $\mathcal{T}_\mathscr{M}\mathcal{A}_Q$. Then, there is a full embedding of categories {\normalfont $\mathcal{T}_\mathscr{M}\mathcal{A}_Q/\mathcal{I}\text{-mod}\hookrightarrow\mathcal{T}_\mathscr{M}\mathcal{A}_Q\text{-mod}$.}
	\end{proposition}
	\begin{proof}
		Let $\pi:\mathcal{T}_\mathscr{M}\mathcal{A}_Q\to \mathcal{T}_\mathscr{M}\mathcal{A}_Q/\mathcal{I}$ be the canonical projection. We claim that the restriction of scalars functor, $\pi^*:\mathcal{T}_\mathscr{M}\mathcal{A}_Q/\mathcal{I}\text{-mod}\to\mathcal{T}_\mathscr{M}\mathcal{A}_Q\text{-mod}$, induced by $\pi$ is the full embedding we are looking for. Indeed, for every $\mathcal{T}_\mathscr{M}\mathcal{A}_Q/\mathcal{I}$-module $\mathcal{M}$, $\pi^*\mathcal{M}$ is the $\mathcal{T}_\mathscr{M}\mathcal{A}_Q$-module whose underlying sheaf is $\mathcal{M}$ and structure morphism $\mu_{\mathcal{T}_\mathscr{M}\mathcal{A}_Q}:\mathcal{T}_\mathscr{M}\mathcal{A}_Q\otimes_{\mathcal{O}_{X}}\mathcal{M}\to \mathcal{M}$ is the sheaf morphism associated to the following composition of presheaf morphisms
		$$
		\begin{matrix}
			\mathcal{T}_\mathscr{M}\mathcal{A}_Q(U)\otimes_{\mathcal{O}_{X}(U)}\mathcal{M}(U) & \longrightarrow & (\mathcal{T}_{\mathscr{M}}\mathcal{A}_Q/\mathcal{I})(U)\otimes_{\mathcal{O}_{X}(U)}\mathcal{M}(U) & \longrightarrow & \mathcal{M}(U)\\
			\sigma\otimes m & \longmapsto & \pi(\sigma)\otimes m & \longmapsto&  \widetilde{\mu}_{\mathcal{T}_\mathscr{M}\mathcal{A}_Q/\mathcal{I}}(\pi(\sigma)\otimes m)
		\end{matrix}.
		$$
		Here, $\widetilde{\mu}_{\mathcal{T}_\mathscr{M}\mathcal{A}_Q/\mathcal{I}}$ is the presheaf morphism associated to the morphism $\mu_{\mathcal{T}_\mathscr{M}\mathcal{A}_Q/\mathcal{I}}$ giving $\mathcal{M}$ structure of $\mathcal{T}_\mathscr{M}\mathcal{A}_Q/\mathcal{I}$-module. At the level of morphisms, if $f:\mathcal{M}\to \mathcal{N}\in \text{Hom}_{\mathcal{T}_\mathscr{M}\mathcal{A}_Q/\mathcal{I}}(\mathcal{M},\mathcal{N})$, then $\pi^*f:\pi^*\mathcal{M}\to \pi^*\mathcal{N}\in \text{Hom}_{\mathcal{T}_\mathscr{M}\mathcal{A}_Q}(\pi^*\mathcal{M},\pi^*\mathcal{N})$ is the morphism of $\mathcal{T}_\mathscr{M}\mathcal{A}_Q$-modules defined by the rules
		$$
		\pi^*f(m_1+m_2)=f(m_1+m_2)=f(m_1)+f(m_2)=\pi^*f(m_1)+\pi^*f(m_2)
		$$
		and 
		$$
		\pi^*f\circ\widetilde{\mu}_{\mathcal{T}_\mathscr{M}\mathcal{A}_Q}(\sigma\otimes m_1)=f\circ \widetilde{\mu}_{\mathcal{T}_\mathscr{M}\mathcal{A}_Q/\mathcal{I}}(\pi(\sigma)\otimes m_1)
		$$
		for all $\sigma\in\mathcal{T}_\mathscr{M}\mathcal{A}_Q(U)$, $m_1,m_2\in \mathcal{M}(U)$ and $U\subseteq X$ open set. 
		\\
		\\
		It is straightforward to see, from how we have defined the functor $\pi^*$, that the claimed statements hold.
	\end{proof}
	
	\begin{remark}\label{remarkFullEmbeddingQuotients}
		Note that the previous proposition still holds true if we replace the twisted path algebra $\mathcal{T}_\mathscr{M}\mathcal{A}_Q$ by any other sheaf of $\mathcal{O}_X$-algebras. 
	\end{remark}
	The following result generalizes Proposition \ref{equivCategoriesTwistedModulesReps} to the case of representations with relations. It is also a generalization of Bruzzo, Bartocci and Rava's result for which the twisting is trivial \cite[Theorem 4.2]{BartocciHomologyTwistedQuiverBundles}.
	\begin{theorem}\label{equivalenceOfCatTwistedRepresentationsModulesWithRelations}
		Let $\mathcal{R}$ be a set of relations and $\mathcal{I}_\mathcal{R}$ be the two-sided ideal of the twisted path algebra $\mathcal{T}_\mathscr{M}\mathcal{A}_Q$ generated by this set.  There is an equivalence of categories between {\normalfont$\text{Rep}(\mathscr{M}Q,\mathcal{I}_\mathcal{R})$} and {\normalfont $\mathcal{T}_\mathscr{M}\mathcal{A}_Q/\mathcal{I}_\mathcal{R}$-mod} compatible with that in Proposition \ref{equivCategoriesTwistedModulesReps} in the sense that the diagram of functors {\normalfont
			\begin{center}
				\begin{tikzcd}[column sep =large, row sep =large]
					\text{Rep}(\mathscr{M}Q,\mathcal{I}_\mathcal{R})\arrow[r,""] \arrow[d,',hookrightarrow,""] & \mathcal{T}_\mathscr{M}\mathcal{A}_Q/\mathcal{I}_\mathcal{R}\text{-mod}\arrow[d,hookrightarrow,"\pi^*"] \\
					\text{Rep}(\mathscr{M}Q) \arrow[r, phantom, "\simeq", centered] \arrow[r,',bend left=10,"G"'] & \mathcal{T}_\mathscr{M}\mathcal{A}_Q\text{-mod} \arrow[l,',bend left=10,"F"']
				\end{tikzcd}
		\end{center}}
		\noindent commutes (up to isomorphism). Here, $\pi^*$ is the restriction of scalars functor in Proposition \ref{fullEmbeddingPathAlgebras} and $F,G$ are the functors defined in Proposition \ref{equivCategoriesTwistedModulesReps}.
	\end{theorem}
	\begin{proof}
		Let $((\mathcal{E}_i)_{i\in Q_0},(\varphi_{\alpha})_{\alpha\in Q_1})$ be a representation of the quiver $Q$ that satisfies the relations in $\mathcal{I}_\mathcal{R}$. We now endow $\mathcal{E}=\bigoplus_{i\in Q_0}\mathcal{E}_i$ with structure of $\mathcal{T}_\mathscr{M}\mathcal{A}_Q/\mathcal{I}_\mathcal{R}$-module. As we saw in Proposition \ref{equivCategoriesTwistedModulesReps}, the functor $G$ endows $\mathcal{E}$ with structure of $\mathcal{T}_\mathscr{M}\mathcal{A}_Q$-module via a map $\mathcal{T}_\mathscr{M}\mathcal{A}_Q\to\mathcal{E}nd_{\mathcal{O}_X}(\mathcal{E})$. Clearly every section of the ideal $\mathcal{I}_\mathcal{R}$ is in the kernel of this morphism so we get a map $\mathcal{T}_\mathscr{M}\mathcal{A}_Q/\mathcal{I}_\mathcal{R}\to \mathcal{E}nd_{\mathcal{O}_X}(\mathcal{E})$ hence endowing $\mathcal{E}$ with structure of $\mathcal{T}_\mathscr{M}\mathcal{A}_Q/\mathcal{I}_\mathcal{R}$-module. Now, let $(f_i:\mathcal{E}_i\to \mathcal{E}'_i)_{i\in Q_0}$ be a morphism of $\mathscr{M}$-twisted representations. Arguing similarly as we did in Proposition \ref{equivCategoriesTwistedModulesReps} we can show that the map $\oplus_{i\in Q_0}f_i:\mathcal{E}\to \mathcal{E}'$ is $\mathcal{T}_\mathscr{M}\mathcal{A}_Q/\mathcal{I}_\mathcal{R}$-linear.  
		\\
		\\
		Let us now give a quasi-inverse for the functor defined above. Let $\mathcal{\mathcal{E}}$ be a $\mathcal{T}_\mathscr{\mathscr{M}}\mathcal{A}_Q/\mathcal{I}_{\mathcal{R}}$-module. Then, $F\circ \pi^*(\mathcal{E})=((\mathcal{E}_i)_{i\in Q_0},(\varphi_\alpha)_{\alpha\in Q_1})$ is a $\mathscr{M}$-twisted representation of $Q$. Let $\sum m_p\in \mathcal{I}_\mathcal{R}(U)$ be a generator, then 
		$$
		\bigg( \sum \psi_{m_p}\bigg)(s_{tp})= \widetilde{\mu}\bigg(\sum \pi(m_p)\otimes s_{tp}\bigg)=0
		$$
		for $s_{tp}\in \mathcal{E}_{tp}(V)$, $V\subseteq U\subseteq X$ an open subset and $\widetilde{\mu}$ the presheaf morphism corresponding to the structure morphism $\mu:\mathcal{T}_\mathscr{M}\mathcal{A}_Q/\mathcal{I}_\mathcal{R}\otimes_{\mathcal{O}_{X}}\mathcal{E}\to \mathcal{E}$. We have just shown that $F\circ \pi^*(\mathcal{\mathcal{E}})$ is in fact a representation in $\text{Rep}(\mathscr{M}Q,\mathcal{I}_\mathcal{R})$. The functor $F\circ \pi^*$ is the desired quasi-inverse.
	\end{proof}
	\subsection{Tensor products of twisted quiver representations}\label{tensorProductsTwistedRepresentations}
	Let us put all together to present our notion of tensor product of twisted quiver representations. Let $\mathcal{E}^{'} (\mathcal{E}'')$ be a $\mathscr{M}'(\mathscr{M}'')$-twisted representation of a quiver $Q'(Q'')$. Because of the equivalence of categories in Proposition \ref{equivalenceOfCatTwistedRepresentationsModulesWithRelations}, these representations correspond to modules, say $\mathcal{M}_{\mathcal{E}'}$ and $\mathcal{M}_{\mathcal{E}''}$, over the corresponding twisted path algebras. The tensor product $\mathcal{M}_{\mathcal{E}'}\otimes _{\mathcal{O}_X}\mathcal{M}_{\mathcal{E}''}$ can be given structure of $\mathcal{T}_{\mathscr{M}'}\mathcal{A}_{Q'}\otimes_{\mathcal{O}_X}\mathcal{T}_{\mathscr{M}''}\mathcal{A}_{Q''}$-module as shown in Equation (\ref{tensorProductOfModulesOverPathAlgebras}). Now, Proposition \ref{isomorphismTensorProductTwistedPathAlgebras} tells us that we can think of this tensor product as a module over $\mathcal{T}_{\mathscr{M}}\mathcal{A}_{Q'\otimes Q''}/\mathcal{I}$ where $\mathcal{I}$ is the two-sided ideal generated by some relations on the tensor quiver $Q'\otimes Q''$. Finally, Theorem \ref{equivalenceOfCatTwistedRepresentationsModulesWithRelations} shows that the $\mathcal{T}_\mathscr{M}\mathcal{A}_{Q'\otimes Q''}/\mathcal{I}$-module $\mathcal{M}_{\mathcal{E}'}\otimes _{\mathcal{O}_X}\mathcal{M}_{\mathcal{E}''}$ corresponds to a representation of the tensor quiver $Q'\otimes Q''$ with relations, in the sense of Definition \ref{defTwistedRepresentationWithRelations}, which we can explicitly write down from the definition of $\mathcal{I}$. Thus it makes sense to define the \emph{tensor product}, $\mathcal{E}'\otimes \mathcal{E}''$, of the representations $\mathcal{E}'$ and $\mathcal{E}''$ as this representation of the tensor quiver with relations. A similar construction, now using Theorem \ref{isomorphismTensorProductTwistedPathAlgebrasWithRelations}, can be carried out when we are tensoring twisted quiver representations with relations. 
	\subsection{Untwisted representations} \label{theUntwistedCase} When the twisting modules are trivial, that is when $\mathscr{M}_\alpha=\mathcal{O}_X$ for all $\alpha\in Q_1$, the theory simplifies and we have more down to earth descriptions of the objects introduced in the previous sections. In this case, representations of $Q$ are given by tuples $\mathcal{E}=((\mathcal{E}_i)_{i\in Q_0},(\varphi_\alpha:\mathcal{E}_{t\alpha}\to \mathcal{E}_{h\alpha})_{\alpha\in Q_1})$ with $\mathcal{E}_i$ being $\mathcal{O}_X$-modules and $\varphi_\alpha$ morphisms. The path algebra, which we denote by $\mathcal{A}_Q$, is given by 
	$$
	\mathcal{A}_Q=\bigoplus_{p \text{ path}}\mathcal{O}_X\cdot p
	$$
	together with the obvious product. 
	\\
	\\
	The equivalence of categories in Proposition \ref{equivCategoriesTwistedModulesReps} can be easily described. Starting from a quiver representation  $\mathcal{E}=((\mathcal{E}_i)_{i\in Q_0},(\varphi_\alpha)_{\alpha\in Q_1})$, the path algebra $\mathcal{A}_Q$ acts on $G(\mathcal{E})$ by:  $\alpha_1\cdots\alpha_k\cdot m := \iota_{h\alpha_1}\circ\varphi_{\alpha_1}\circ\ldots\circ\varphi_{\alpha_k}\circ\pi_{t\alpha_k}(m)$ with $\alpha_1\cdots\alpha_k$ a path, $m\in G(\mathcal{\mathcal{E}})(U)$, $U\subseteq X$ an open set, and $\iota_i:\mathcal{E}_i\hookrightarrow G(\mathcal{M})$ and $\pi_i:G(\mathcal{M})\to \mathcal{E}_i$ the canonical inclusions and projections respectively. In this way one endows $G(\mathcal{E})$ with structure of $\mathcal{A}_Q$-module. Conversely, given a $\mathcal{A}_Q$-module $\mathcal{M}$, $F(\mathcal{M})=((\mathcal{F}_i)_{i\in Q_0},(\psi_\alpha)_{\alpha\in Q_1})$ is the representation such that $\mathcal{F}_i=\mu(\mathcal{O}_X\cdot e_i\otimes_{\mathcal{O}_{X}}\mathcal{M})$ and $\psi_\alpha$ is the restriction of $\mu$ to $\mathcal{O}_X\cdot \alpha\otimes_{\mathcal{O}_{X}}\mathcal{F}_{t\alpha}$ with $\mu:\mathcal{A}_Q\otimes_{\mathcal{O}_{X}}\mathcal{M}\to \mathcal{M}$ the structure morphism. 
	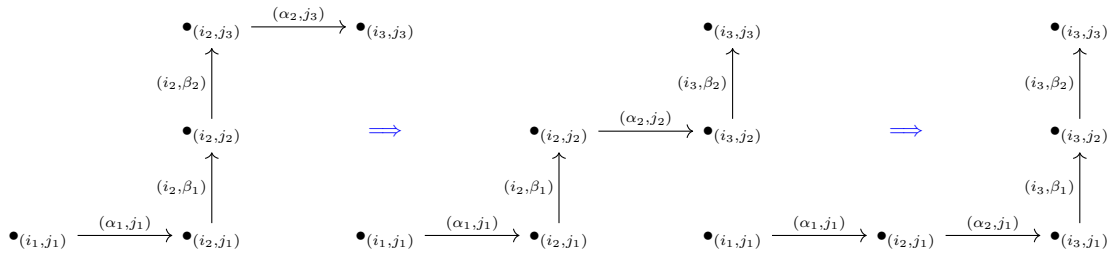
\begin{figure}[H]
		\begin{center}
			\begin{tikzcd}[scale cd=0.8, column sep =large, row sep =large]
				&\bullet_{(i_2,j_3)}\arrow[r,"{(\alpha_2,j_3)}"]&\bullet_{(i_3,j_3)}& & \bullet_{(i_3,j_3)} &&\bullet_{(i_3,j_3)}\\
				& \bullet_{(i_2,j_2)}\arrow[u,"{(i_2,\beta_2)}"]&\textcolor{blue}{\Longrightarrow} &\bullet_{(i_2,j_2)}\arrow[r,"{(\alpha_2,j_2)}"]&\bullet_{(i_3,j_2)}\arrow[u,"{(i_3,\beta_2)}"]&\textcolor{blue}{\Longrightarrow} &\bullet_{(i_3,j_2)}\arrow[u,"{(i_3,\beta_2)}"] \\
				\bullet_{(i_1,j_1)}\arrow[r,"{(\alpha_1,j_1)}"]& \bullet_{(i_2,j_1)}\arrow[u,"{(i_2,\beta_1)}"]&  \bullet_{(i_1,j_1)}\arrow[r,"{(\alpha_1,j_1)}"]&\bullet_{(i_2,j_1)}\arrow[u,"{(i_2,\beta_1)}"] &\bullet_{(i_1,j_1)}\arrow[r,"{(\alpha_1,j_1)}"]&\bullet_{(i_2,j_1)}\arrow[r,"{(\alpha_2,j_1)}"]&\bullet_{(i_3,j_1)} \arrow[u,"{(i_3,\beta_1)}"] 
			\end{tikzcd}
		\end{center}
		\caption{\label{equivalenceOfPaths} Equivalent paths}
	\end{figure}
	The ideal $\mathcal{I}$ in Proposition \ref{isomorphismTensorProductTwistedPathAlgebras} is generated, for the untwisted case, by all the differences
	$$
	(h\alpha,\beta)(\alpha,t\beta)-(\alpha,h\beta)(t\alpha,\beta)
	$$
	with $\alpha\in Q'_1$ and $\beta\in Q_1''$ and the isomorphism $\mathcal{A}_{Q'\otimes Q''}/\mathcal{I}\cong \mathcal{A}_{Q'}\otimes_{\mathcal{O}_X}\mathcal{A}_{Q''}$ can be seen more ``geometrically''. For instance, Figure \ref{equivalenceOfPaths} shows  three equivalent paths modulo $\mathcal{I}$ and we note that the right-most one is a path in the set $S$ defined in the proof of the aforementioned proposition. 
	\\
	\\
	Thus, for $\mathcal{E}=((\mathcal{E}_i)_{i\in Q'_0},(\varphi_\alpha)_{\alpha\in Q'_1}), \mathcal{F}=((\mathcal{F}_j)_{j\in Q''_0},(\psi_\beta)_{\beta\in Q''_1})$, the tensor product $\mathcal{E}\otimes \mathcal{F}$ will be a representation with relations of the quiver $Q'\otimes Q''$ as shown below. 
	
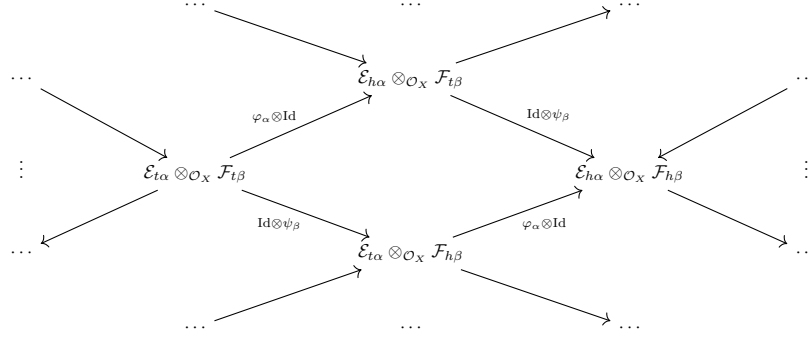
\begin{center}
		\begin{tikzcd}[scale cd=0.7, column sep =large, row sep =normal]
			&\cdots\arrow[dr]&\cdots&\cdots&\\
			\cdots\arrow[dr]& & \mathcal{E}_{h\alpha}\otimes_{\mathcal{O}_X}\mathcal{F}_{t\beta} \arrow[dr,,"\text{Id}\otimes\psi_{\beta}"]\arrow[ur]&& \cdots \arrow[dl]\\
			\vdots& \mathcal{E}_{t\alpha}\otimes_{\mathcal{O}_X}\mathcal{F}_{t\beta}\arrow[dl]\arrow[ur,"\varphi_\alpha\otimes \text{Id}"] \arrow[dr,"\text{Id}\otimes \psi_{\beta}",']& & \mathcal{E}_{h\alpha}\otimes_{\mathcal{O}_X}\mathcal{F}_{h\beta}\arrow[dr]& \vdots \\
			\cdots &  &\mathcal{E}_{t\alpha}\otimes_{\mathcal{O}_X}\mathcal{F}_{h\beta}\arrow[dr]\arrow[ur,"\varphi_{\alpha}\otimes\text{Id}",'] &  & \cdots\\
			&\cdots\arrow[ur]&\cdots&\cdots& 
		\end{tikzcd}
	\end{center}
	Following Remark \ref{representationsInOtherCategories}, if $X=\text{Spec}(\mathbb{C})$, then $\text{Coh}(X)=\text{Vect}_\mathbb{C}$, that is, the category of finite dimensional vector spaces over the field of the complex numbers. So, in this case, untwisted representations are precisely the classical quiver representations and our notion of tensor product recovers that of Herschend \cite{HerschendTensorProducts} and Das, Dubey and Raghavendra \cite{DasEtAlTensors}.
	
	\section{Moduli of quiver representations and the quiver vortex equations: A brief survey}\label{surveySection}
	In view of some applications we will now consider either (twisted) quiver representations in the category of locally free sheaves (holomorphic vector bundles) over a compact Kähler manifold $X$ or the classical linear case when $X=\text{Spec}(\mathbb{C})$.
	\subsection{Generalities on (twisted) quiver bundles} 
	\begin{definition}
		Let $Q=(Q_0,Q_1,h,t)$ be a quiver and $M=\{M_\alpha\}_{\alpha\in Q_1}$ a collection of holomorphic vector bundles. A \emph{holomorphic $M$-twisted quiver bundle}, or simply, a \emph{$Q$-bundle}, is a $M$-twisted representation of the quiver $Q$ in the category of holomorphic vector bundles over $X$. In other words, a $Q$-bundle is a tuple $\mathcal{E}=((\mathcal{E}_i)_{i\in Q_0},(\varphi_\alpha)_{\alpha\in Q_1})$ where $\mathcal{E}_i\to X$ is a holomorphic vector bundle and $\varphi_\alpha:M_\alpha\otimes \mathcal{E}_{t\alpha}\to \mathcal{E}_{h\alpha}$ is a morphism of vector bundles. 
	\end{definition}
	We call 
	\begin{equation}\label{equationTypeQBundle}
		t(\mathcal{E})=(\text{rank}(\mathcal{E}_i),\deg(\mathcal{E}_i))_{i\in Q_0}\in \mathbb{N}^{|Q_0|}\times \mathbb{Z}^{|Q_0|}
	\end{equation}
	the \emph{type}  of the quiver bundle $\mathcal{E}$ and we note that this is independent of the morphisms $(\varphi_\alpha)_{\alpha\in Q_1}$. There is a natural notion of \emph{$Q$-subbundle}, this is given by a tuple $\mathcal{F}=(\mathcal{F}_i)_{i\in Q_0}$ with $\mathcal{F}_i$ a vector subbundle of $\mathcal{E}_i$ for all $i\in Q_0$ such that $\varphi_{\alpha}(M_\alpha\otimes \mathcal{F}_{t\alpha})\subseteq\mathcal{F}_{h\alpha}$ for all $\alpha\in Q_1$.  Given $\sigma=(\sigma_i)_{i\in Q_0}\in \mathbb{R}^{|Q_0|}_{>0}$, $\theta=(\theta_i)_{i\in Q_0}\in \mathbb{R}^{|Q_0|}$, the \emph{$(\sigma,\theta)$-slope} of a non-zero quiver bundle $\mathcal{E}$ is defined as 
	\begin{equation}\label{slopeStabilityQuiverBundles}
		\mu_{\sigma,\theta}(\mathcal{E})=\frac{\sum_{i\in Q_0}\sigma_i\deg(\mathcal{E}_i)+\theta_i\text{rank}(\mathcal{E}_i)}{\sum_{i\in Q_0}\text{rank}(\mathcal{E}_i)}. 
	\end{equation}
	
	\begin{definition}\label{defStableQuiverBundle}
		A $Q$-bundle is said to be $(\sigma,\theta)$-\emph{(semi)stable} if for every, non-trivial, quiver sub-bundle $\mathcal{F}$ of $\mathcal{E}$, $\mu_{\sigma,\theta}(\mathcal{F})(\leq)<\mu_{\sigma,\theta}(\mathcal{E})$. We say that a $Q$-bundle is $(\sigma,\theta)$-polystable if it can be written as finite direct sum of $(\sigma,\theta)$-stable $Q$-bundles all of them of the same $(\sigma,\theta)$-slope.
	\end{definition}
	
	\begin{remark}\label{invarianceByTranslationStability}
		Slope (semi)stability is invariant under translations, that is, a quiver bundle $\mathcal{E}$ is $(\sigma,\theta)$-(semi)stable if and only if it is $(\sigma',\theta')$-(semi)stable for $\sigma=\sigma'$, $\theta'=(\theta_i+\beta\sigma_i)_{i\in Q_0}$ and $\beta\in \mathbb{R}$.
	\end{remark}
	\subsection{The linear case}\label{theLinearCaseSubsection} In this section we deal with the case $X=\text{Spec}(\mathbb{C})$ and we closely follow \cite{KingModuli} and \cite{HoskinsSurvey}. As hinted before, since vector bundles over points are just vector spaces, representations of a quiver $Q$ are given by tuples $((V_i)_{i\in Q_0},(\varphi_\alpha)_{\alpha\in Q_1})$ where $V_i$ is a finite dimensional complex vector space and $\varphi_{\alpha}:V_{t\alpha}\to V_{h\alpha}$ is a linear map for all $i\in Q_0$ and $\alpha\in Q_1$. The tuple $d=(d_i:=\dim(V_i))_{i\in Q_0}\in \mathbb{N}^{|Q_0|}$ is known as the \emph{dimension vector} of the representation and we note that this is just the type of the representation in the sense of Equation (\ref{equationTypeQBundle}). The space of all representations of the quiver $Q$, with fixed dimension vector $d\in \mathbb{N}^{|Q_0|}$, is then parameterised by the affine space
	$$
	\text{Rep}(Q,d)=\bigoplus_{\alpha\in Q_1}\text{Hom}(\mathbb{C}^{t\alpha},\mathbb{C}^{h\alpha}).
	$$
	The reductive group $\text{GL}(d)=\prod_{i\in Q_0}\text{GL}(d_i)$ acts algebraically on this parameter space by conjugation:
	\begin{equation}\label{actionGLOnQuiverReps}
		g\cdot \varphi=(g_{h\alpha}\varphi_\alpha g_{t\alpha}^{-1})_{\alpha\in Q_1}
	\end{equation}
	for all $g:=(g_i)_{i\in Q_0}\in \text{GL}(d)$ and $\varphi:=(\varphi_\alpha)_{\alpha\in Q_1}\in \text{Rep}(Q,d)$. Notice that the orbits for this action correspond to isomorphism classes of $d$-dimensional quiver representations.
	\\
	\\
	We can specialize the (semi)stability condition in Equation (\ref{slopeStabilityQuiverBundles}) and Definition \ref{defStableQuiverBundle} to the linear case at hand. Notice that the degree terms in Equation (\ref{slopeStabilityQuiverBundles}) vanish, again because we are dealing with vector spaces, so the stability condition only depends on $\theta$ which in the algebraic category needs to be an element in $\mathbb{Z}^{|Q_0|}$. Following Remark \ref{invarianceByTranslationStability}, one can show that $\theta$-(semi)stability is equivalent to $\theta'$-(semi)stability for $\theta'\in\mathbb{Z}^{|Q_0|}$ given by
	$$
	\theta_i':=\theta_i\sum_{j\in Q_0}d_j-\sum_{j\in Q_0}\theta_jd_j.
	$$
	The parameter $\theta$ determines a character 
	\begin{equation}\label{character}
		\begin{matrix}
			\chi_\theta: & \text{GL}(d) & \longrightarrow & \mathbb{G}_m \\
			&(g_i)_{i\in Q_0}&\longmapsto& \prod_{i\in Q_0}\det g_i^{-\theta'_i}
		\end{matrix}
	\end{equation}
	which we use to linearize the action of $\text{GL}(d)$ on the parameter space $V=\text{Rep}(Q,d)$. By this we mean a lift of the $\text{GL}(d)$-action to the line bundle $\mathcal{L}_\theta:=V\times \mathbb{C}$ given by 
	$$
	g\cdot(\varphi,z)=(g\cdot\varphi,\chi_\theta(g)^{-1}z ).
	$$
	Invariant sections of $\mathcal{L}_\theta$ can be used to construct moduli spaces of $\theta$-(semi)stable quiver representations using GIT. One can show that 
	$$\mathcal{O}_V^{\chi_\theta^j}:=H^0(\text{Rep}(Q,d),\mathcal{L}_\theta^{\otimes j})^{\text{GL}(d)}=\{f\in \mathcal{O}_V| \ f(g\cdot\varphi) = \chi_\theta(g)^jf(\varphi) \forall\varphi\in \text{Rep}(Q,d) \text{ and }g\in \text{GL}(d)\}.$$ 
	Since the subgroup of $\text{GL}(d)$ given by 
	$$
	\Delta:= \{(t\text{Id}_{d_i})_{i\in Q_0}|\ t\in \mathbb{G}_m\}\cong\mathbb{G}_m
	$$
	acts trivially on the parameter space $\text{Rep}(Q,d)$, invariant sections of the line bundle $\mathcal{L}_\theta^{\otimes j}$ exist if and only if $\chi_\theta(\Delta)=1$. This holds by construction of $\theta'$ since $\theta'\cdot d=0$.
	\\
	\\
	The $\theta$-(semi)stability condition in Equation (\ref{slopeStabilityQuiverBundles}) coindides with the GIT $\chi_\theta$-(semi)stability condition \cite[Proposition 3.1]{KingModuli} so we have a chain of open inclusions $\text{Rep}(Q,d)^{\theta-s}\subseteq\text{Rep}(Q,d)^{\theta-ss}\subseteq\text{Rep}(Q,d)$. The GIT quotient
	\begin{equation}\label{GITQuotientLinearCase}
		\pi:\text{Rep}(Q,d)^{\theta-ss}\to \mathcal{M}^{\theta-ss}(Q,d):=\text{Rep}(Q,d)^{\theta-ss}\sslash\text{GL}(d)=\text{Proj}(\bigoplus_{j\geq 0}\mathcal{O}_V^{\chi_\theta^j})
	\end{equation}
	is a good quotient for the $\text{GL}(d)$-action on $\text{Rep}(Q,d)^{\theta-ss}$. It parameterises the closed orbits in $\text{Rep}(Q,d)^{\theta-ss}$ which are those corresponding to $\theta$-polystable representations.  Furthermore, $\pi$ restricts to a geometric quotient $\pi|_{\theta-s}:\text{Rep}(Q,d)^{\theta-s}\to \mathcal{M}^{\theta-s}(Q,d):=\text{Rep}(Q,d)^{\theta-s}/\text{GL}(d)$ on the stable set.  The latter quotient, if non-empty, is smooth so as open dense in $\mathcal{M}^{\theta-ss}(Q,d)$ \cite[Theorem 10.8]{libroKirillov}. The quotient $\mathcal{M}^{\theta-(s)s}(Q,d)$ is known as \emph{the moduli space of $d$-dimensional $\theta$-(semi)stable representations of $Q$}.
	\\
	\\
	It is also worth mentioning that we have an injection of algebras $\mathcal{O}_V^{\chi_\theta^0}=\mathcal{O}_V^{\text{GL}(d)}\hookrightarrow \bigoplus_{j\geq 0}\mathcal{O}_V^{\chi_\theta^j}$ which in turn induces a morphism of schemes 
	$$
	\text{Rep}(Q,d)^{\theta-ss}\sslash\text{GL}(d)\to \text{Rep}(Q,d)\sslash \text{GL}(d)=\text{Spec}(\mathcal{O}_{\text{Rep}(Q,d)}^{\text{GL}(d)})
	$$
	which is projective \cite[p.599]{ReinekeSurvey}. If $Q$ is a quiver without oriented cycles, then the GIT quotient $\text{Rep}(Q,d)\sslash\text{GL}(d)$ will be a point and therefore the moduli space $\mathcal{M}^{\theta-ss}(Q,d)$ will be a projective variety. This is a consequence of a classic result of Le Bruyn and Procesi \cite[Theorem 1]{LeBruynProcesi}.
	\begin{example}\label{exampleMatricesUnderConjugation}
		Let us consider the Jordan quiver, $Q_J$, from Figure \ref{JordanQuiver}. Note that we can identify $\text{Rep}(Q_J,d)\cong \text{End}(\mathbb{C}^d)$ so the action of $\text{GL}(d)$ on this space is given by conjugation of matrices. We now discuss on the moduli space of $d$-dimensional representations of the Jordan quiver for the stability parameter $\theta=0$ (this is, in fact, the only admissible one). As every polynomial invariant under the action of $\text{GL}(d)$ is completely determined by its values on diagonal matrices, one can see that
		$$
		\mathcal{M}(Q_J,d):=\mathcal{M}^{\theta-ss}(Q_J,d)=\text{End}(\mathbb{C}^d)\sslash \text{GL}(d)=\text{Spec}(\mathbb{C}[(a_{ij})_{i,j=1,\ldots,d}]^{\text{GL}(d)})=\text{Spec}(\mathbb{C}[\lambda_1,\ldots,\lambda_d]^{\mathfrak{S}_d})
		$$
		where the symmetric group in $d$ elements, $\mathfrak{S}_d$, acts on the ring $\mathbb{C}[\lambda_1,\ldots,\lambda_d]$ by permuting variables. Furthermore, the fundamental theorem of invariant theory tells us that $\mathbb{C}[\lambda_1,\ldots,\lambda_d]^{\mathfrak{S}_d}\cong \mathbb{C}[\text{Tr}(A),\ldots,\text{det}(A)]$, that is, the ring of polynomials in the variables given by the coefficients of the characteristic polynomial of $A=(a_{ij})_{i,j=1,\ldots,d}\in \text{End}(\mathbb{C}^d)$\ \cite[p.18]{KraftProcesi}. Thus,  
		$$
		\mathcal{M}(Q_J,d)\cong \mathbb{A}^d_\mathbb{C}.
		$$
	\end{example}
	\begin{example}\label{exampleProjectiveSpaceAsQuiverModuli}
		Let $Q_n$ be the Kronecker quiver with $n$ edges (see Example \ref{kroneckerQuiverExample}). For the dimension vector $d=(1,1)$ we can identify $\text{Rep}(Q,d)\cong \mathbb{A}^n_\mathbb{C}$ and the action of $\text{GL}(d)=\mathbb{G}_m\times \mathbb{G}_m$ on this space, following Equation (\ref{actionGLOnQuiverReps}), will be given by $(t_1,t_2)\cdot z=\text{diag}(t_2t_1^{-1},\ldots,t_2t_1^{-1})z$. Also, we note that $\mathcal{O}_{\text{Rep}(Q,d)}\cong \mathbb{C}[z_1,\ldots,z_n]$.
		\newline
		\newline
		We now study the GIT quotients and semistable loci for $\theta'=(1,-1)$. Fix a representation $z \in \text{Rep}(Q,d)$ and observe that if $z\neq 0$, there is only one non-trivial and proper subrepresentation which has dimension vector $(0,1)$. On the other hand, if $z=0$ we have proper and non-trivial subrepresentations of dimension vectors $(0,1)$ and $(1,0)$. Hence, $\text{Rep}(Q,d)^{\theta'-ss}=\text{Rep}(Q,d)^{\theta'-s}=\mathbb{A}^n_\mathbb{C}\setminus\{0\}$.
		\\
		\\
		Finally, we will give an explicit description of the GIT quotient $\mathcal{M}(Q,d)^{\theta'-ss}$. We observe that the character $\chi_{\theta'}:\text{GL}_d\to \mathbb{G}_m$ is given by $(t_1,t_2)\mapsto t_2/t_1$ and therefore the action of $\text{GL}(d)$ on the line bundle $\mathcal{L}_{\theta'}$ will be given by $(t_1,t_2)\cdot (x,y)=((t_1,t_2)\cdot x,(t_1/t_2)y)$. For all $j\geq 0$, $$\mathcal{O}_{\mathbb{A}^n_\mathbb{C}}^{\chi_{\theta'}^j}=\{f\in \mathcal{O}_{\text{Rep}(Q,d)}\cong \mathbb{C}[z_1,\ldots,z_n]\ |\ f \text{ homogeneous of degree }j\}.$$ 
		Thus, 
		$$
		\mathcal{M}(Q,d)^{\theta'-s}=\text{Proj}(\mathbb{C}[z_1,\ldots,z_n])=\mathbb{P}_\mathbb{C}^{n-1}
		$$
		and the map $\pi:\text{Rep}(Q,d)^{\theta'-s}\to	\mathcal{M}(Q,d)^{\theta'-s}$ is just the canonical projection. 
	\end{example}
	There is a symplectic point of view on these moduli spaces that we now discuss.  For all $i\in Q_0$, fix an hermitian inner product on $\mathbb{C}^{d_i}$.  Now, on each direct summand of $\text{Rep}(Q,d)$ we consider the standard hermitian  inner product given by 
	\begin{equation}\label{standardHermitianPairing}
		\begin{matrix}
			\langle A , B \rangle = \text{Tr}(AB^*).
		\end{matrix}
	\end{equation} 
	An inner product on the parameter space $\text{Rep}(Q,d)$ is then given by
	\begin{equation}\label{hermitianInnerProductParameterSpaceQuiverReps}
		\begin{matrix}
			H: & \text{Rep}(Q,d)\times \text{Rep}(Q,d)&\longrightarrow & \mathbb{C} \\
			&((\varphi_\alpha)_{\alpha\in Q_1},(\psi_\alpha)_{\alpha\in Q_1})&\longmapsto & \sum_{\alpha\in Q_1}\text{Tr}(\varphi_{\alpha}\psi_\alpha^*).
		\end{matrix}
	\end{equation}
	The group $\text{U}(d)=\prod_{i\in Q_0}\text{U}(d_i)$, with Lie algebra $\mathfrak{u}(d)=\prod_{i\in Q_0}\mathfrak{u}(d_i)$, is a maximal compact subgroup of $\text{GL}(d)$ that preserves this hermitian form. 
	Furthermore, the inner product $H$ induces a Kähler structure on $\text{Rep}(Q,d)$ by taking the Kähler form to be $\omega=-2\text{Im}H$. Notice that the symplectic form $\omega$ is preserved by the action of $\text{U}(d)$ since $H$ is $\text{U}(d)$-invariant. Therefore, the linear action of the reductive group $\text{GL}(d)$ on $\text{Rep}(Q,d)$ restricts to a symplectic linear one on $\text{U}(d)$. 
	A moment map for this action is given by 	
	\begin{equation}
		\begin{matrix}\label{defMomentMapLinearCase}
			\mu: & \text{Rep}(Q,d)&\longrightarrow &\mathfrak{u}^*(d) \\
			&\varphi&\longmapsto& (A\longmapsto\sqrt{-1}H(A\cdot \varphi,\varphi)).
		\end{matrix}
	\end{equation}
	Here, the action of $A=(A_i)_{i\in Q_0}\in \mathfrak{u}(d)$ on $\varphi=(\varphi_{\alpha})_{\alpha\in Q_1}\in \text{Rep}(Q,d)$ is given by 
	\begin{equation}\label{infinitesimalActionLieAlgebraUnitaryGroups}
		A\cdot \varphi = (A_{h\alpha}\varphi_\alpha-\varphi_\alpha A_{t\alpha})_{\alpha\in Q_1}
	\end{equation}
	and notice that this is just infinitesimal action or derivative of the action described in Equation (\ref{actionGLOnQuiverReps}). The moment map can be then written more explicitly as
	\begin{align*}
		\mu_\varphi(A)=\sqrt{-1}H(A\cdot \varphi, \varphi ) & =\sqrt{-1}\sum_{\alpha\in Q_1}\text{Tr}(A_{h\alpha}\varphi_\alpha\varphi_\alpha^*-\varphi_\alpha A_{t\alpha}\varphi_\alpha^*) \\
		&=\sqrt{-1}\sum_{i\in Q_0}\text{Tr}(A_i[\varphi,\varphi^*]_i)
	\end{align*}
	where we have denoted 
	\begin{equation}\label{commutatorMomentMap}
		[\varphi,\varphi^*]_i:=\sum_{h\alpha = i}\varphi_\alpha\varphi_\alpha^*-\sum_{t\alpha=i}\varphi_\alpha^*\varphi_\alpha.
	\end{equation}
	The hermitian pairing in Equation (\ref{standardHermitianPairing})
	induces an isomorphism $\mathfrak{u}(d)\cong \mathfrak{u}(d)^*$ and this is why the moment map is often written as $$
	\mu(\varphi)=-\sqrt{-1}([\varphi,\varphi^*]_i)_{i\in Q_0}
	$$
	for all $\varphi\in \text{Rep}(Q,d)$.
	\\
	\\
	The moment map described above is  unique up to addition of an element $\eta\in\mathfrak{u}(d)^*$ fixed by the coadjoint action of $\text{U}(d)$ on $\mathfrak{u}(d)^*$. Recall that any tuple $\theta=(\theta_i)_{i\in Q_0}\in \mathbb{Z}^{|Q_0|}$ gives a character $\chi_\theta:\text{GL}(d)\to\mathbb{G}_m$ (see Equation (\ref{character})) whose restriction to $\text{U}(d)$ has image in $\text{U}(1)$. Thus the derivative 
	$$\begin{matrix}
		\sqrt{-1}\ \text{d}\chi_\theta|_{\text{U}(d)}:&\mathfrak{u}(d)& \longrightarrow &\mathfrak{u}(1) \\
		&A&\longmapsto&\sqrt{-1 }\ \sum_{i\in Q_0}\text{Tr}(A_i(-\theta_i'\text{ Id})), 
	\end{matrix}$$
	seen as an element of $\mathfrak{u}(d)^*$, is a fixed point for the coadjoint action of $\text{U}(d)$ on $\mathfrak{u}(d)^*$. We have made this digression just to point out that we can consider different fibers $\mu^{-1}(\eta)$ of the moment map given in Equation (\ref{defMomentMapLinearCase}),  with $\eta\in\mathfrak{u}(d)^*$ a fixed point for the coadjoint action, and study the corresponding symplectic reductions $\mu^{-1}(\eta)/\text{U}(d)$. In this regard, King showed the following: 
	\begin{theorem}[\cite{KingModuli}]\label{quiverVortexEqnsLinearCase}
		The set {\normalfont$\mu^{-1}(\sqrt{-1}\ \text{d}\chi_\theta|_{\text{U}(d)})$}, which is the solution space to the equation
		{\normalfont
			\begin{equation}\label{quiverVortexEquationsForAPoint}
				\sum_{h\alpha = i}\varphi_\alpha\varphi_\alpha^*-\sum_{t\alpha=i}\varphi_\alpha^*\varphi_\alpha = \theta_i'\text{Id}
			\end{equation}
		}
		
		\noindent for all $i\in Q_0$, intersects each closed  {\normalfont$\text{GL}(d)$}-orbit in {\normalfont$\text{Rep}^{\theta-ss}(Q,d)$}. Moreover, every such closed orbit is intersected by only one {\normalfont$\text{U}(d)$}-orbit in {\normalfont$\mu^{-1}(\sqrt{-1}\ \text{d}\chi_\theta|_{\text{U}(d)})$}.
	\end{theorem}
	Thus, a quiver representation satisfies Equation (\ref{quiverVortexEquationsForAPoint}), for some choice of hermitian inner products, if and only if it is $\theta$-polystable. In other words, we have a bijection, which is in fact a homeomorphism \cite[Theorem 4.2]{HoskinsStratifications},
	$$
	\mu^{-1}(\sqrt{-1} \ \text{d}\chi_\theta|_{\text{U}(d)})/\text{U}(d)\overset{1:1}{\longleftrightarrow} \text{Rep}(Q,d)^{\theta-ss}\sslash\text{GL}(d).
	$$
	For suitable choices of stability parameter $\theta$ and dimension vector $d$, the symplectic reduction $\mu^{-1}(\sqrt{-1}\ \text{d}\chi_\theta|_{\text{U}(d)})/\text{U}(d)$ will be a smooth symplectic manifold, in fact Kähler.
	\begin{remark}\label{indivisibilityDimensionVector}
		In the words of King \cite[Remark 5.4]{KingModuli}, if we choose the dimension vector $d=(d_i)_{i\in Q_0}\in \mathbb{N}^{|Q_0|}$ to be \emph{indivisble} then there exist an open dense subset of $\mathbb{R}^{|Q_0|}$ for which any stability parameter $\theta$ in this subset is such that the $\theta$-semistable locus coincides with the stable one and  $\mathcal{M}(Q,d)^{\theta-ss}=\mathcal{M}(Q,d)^{\theta-s}$ is smooth. The indisivibility condition on the dimension vector $d\in \mathbb{N}^{|Q_0|}$ means that $\text{gcd}\{d_i|i\in Q_0\}=1$. 
	\end{remark} 
	
	\subsection{Moduli of quiver representations with relations}\label{moduliQuiversWithRelations} We briefly discuss on the moduli spaces of quiver representations with relations. We refer the interested reader to Wilkin \cite[Section 3.1]{Wilkin}. From Definition \ref{definitionRelation} recall that a relation in $Q$, when $X=\text{Spec}(\mathbb{C})$, is a formal sum of the form 
	$$
	\sum_{\substack{p\text{ path, }|p|\geq 2\\ tp=i,\ hp=j}}c_p\cdot p,\ c_p\in \mathbb{C}
	$$
	where all but finitely many terms are zero. Let $\mathcal{R}$ be a set of relations of the quiver $Q$ and $d\in \mathbb{N}^{|Q_0|}$ be a dimension vector. For all $r\in \mathcal{R}$, let $tr,\ hr\in Q_0$ denote, respectively, the common head and tail of all the paths appearing in $r$.  Moreover, every such $r\in \mathcal{R}$, determines a morphism of affine varieties
	$
	\nu_r:\text{Rep}(Q,d)\to \text{Hom}(\mathbb{C}^{d_{tr}},\mathbb{C}^{d_{hr}})
	$
	given by composing and adding linear maps according to the information given by the relation $r$. The $d$-dimensional representations of $Q$ satisfying the relations in $\mathcal{R}$ (notice that this agrees with Definition \ref{defTwistedRepresentationWithRelations}) are then parameterised by the closed affine subvariety 
	$$
	\text{Rep}(Q,d,\mathcal{R}):=\bigcap_{r\in \mathcal{R}}\nu^{-1}_r(0)\subseteq\text{Rep}(Q,d).
	$$
	We remark that, under the equivalence of categories in Theorem \ref{equivalenceOfCatTwistedRepresentationsModulesWithRelations}, this set also parameterises modules over the quotient of the path algebra $\mathcal{A}_Q/\mathcal{I}_\mathcal{R}$ of some fixed dimension.  Also, an easy calculation shows that the variety $\text{Rep}(Q,d,\mathcal{R})$ is invariant under the action of $\text{GL}(d)$ which prompts the following definition:
	\begin{definition}
		For a fixed dimension vector $d\in \mathbb{N}^{|Q_0|}$, we call the GIT quotient 
		$$
		\mathcal{M}^{\theta-ss}(Q,d,\mathcal{R}):=\text{Rep}(Q,d,\mathcal{R})\cap \text{Rep}(Q,d)^{\theta-ss}\sslash\text{GL}(d)
		$$
		the \emph{moduli space of $\theta$-semistable representations of Q satisfying the relations $\mathcal{R}$}.
	\end{definition}
	
	\subsection{The quiver vortex equations}
	In this subsection we assume, for the sake of simplicity, that $X$ is a Riemann surface albeit all the results hold true for more general smooth projective Kähler manifolds. We will closely follow the exposition from Álvarez-Cónsul and García-Prada \cite{AlvarezConsulGarciaPradaHKCorrespondenceQuiversVortices}. Let $Q$ be a quiver, $M=\{M_\alpha\}_{\alpha\in Q_1}$ a fixed collection of holomorphic vector bundles together with fixed hermitian metrics $q_\alpha$ on each $M_\alpha$ and $E=((E_i)_{i\in Q_0},(\phi_\alpha)_{\alpha\in Q_1})$ a $M$-twisted representation of $Q$ in the category of smooth complex vector bundles over $X$.  
	\\
	\\
	For all $i\in Q_0$, we let $H_i$ be an hermitian metric on the vector bundle $E_i\to X$. Let $\mathcal{G}_i$ be the group of unitary gauge transformations of the vector bundle $E_i\to X$ or, in other words, the group of automorphisms $\Phi:E_i\to E_i$ such that 
	$$
	H_i(\Phi(s),\Phi(s'))=H_i(s,s')
	$$
	for all $s,s'\in \Omega^0(X,E_i)$. The standard hermitian product on $\mathfrak{u}(d_i)$, given by that in Equation (\ref{standardHermitianPairing}), induces an $L^2$ inner product on $\text{Lie }\mathcal{G}_i$ which in turn gives an inclusion $\text{Lie }\mathcal{G}_i\subseteq \text{Lie }\mathcal{G}_i^*$.
	\\
	\\
	Let $\mathcal{A}_i$ be the set of unitary connections on the vector bundle $E_i\to X$. The gauge group $\mathcal{G}_i$ acts on $\mathcal{A}_i$ by  
	$$
	\begin{matrix}
		\mathcal{G}_i\times\mathcal{A}_i & \longrightarrow & \mathcal{A}_i \\
		(\Phi,\nabla) & \longmapsto & \Phi\circ\nabla\circ\Phi^{-1}
	\end{matrix}
	$$
	so we have, in particular, a natural action of the group $\mathcal{G}=\prod_{i\in Q_0}\mathcal{G}_i$ on the space of unitary connections $\mathcal{A}=\prod_{i\in Q_0}\mathcal{A}_i$. Given two unitary connections $\nabla_1,\nabla_2$ on the vector bundle $E_i\to X$, we have that $\nabla_1-\nabla_2\in \Omega^1(X,\text{ad}(E_i))$. Thus,  the tangent space to $\mathcal{A}_i$, seen as an infinite-dimensional manifold, can be identified with the vector space $\Omega^1(X,\text{ad}(E_i))$. Atiyah and Bott \cite{AtiyahBott} showed that 
	$$
	\omega_i(\xi,\eta)=\int_X\Lambda\text{Tr}(\xi\wedge\eta),
	$$
	with $\xi, \eta\in \Omega^1(X,\text{ad}(E_i))$ and $\Lambda:\Omega^{l,k}(X)\to \Omega^{l-1,k-1}(X)$ the contraction operator with respect to a fixed Kähler form on $X$, defines a symplectic form on $\mathcal{A}_i$. Moreover, they showed the existence of a moment map for the action of the gauge group on the space of unitary connections. This is given by  
	$$
	\begin{matrix}
		\mu_i: & \mathcal{A}_i & \longrightarrow & \text{Lie }\mathcal{G}_i^* \\
		& \nabla_i & \longmapsto & \Lambda F_i.
	\end{matrix}
	$$
	Here, $F_i\in \Omega^{2}(X,\text{ad}(E_i))$ is the curvature of the unitary connection $\nabla_i$ and thus we have that $\Lambda F_i\in \Omega^0(X,\text{ad}(E_i))=\text{Lie}\ \mathcal{G}_i$ which is seen as an element in $\text{Lie }\mathcal{G}_i^*$ by means of the inclusion previously described.
	\\
	\\
	Let 
	$$\mathcal{R}ep(Q,E)=\Omega^0\bigg(X,\bigoplus_{\alpha\in Q_1}\text{Hom}(M_\alpha\otimes E_{t\alpha},E_{h\alpha})\bigg).$$
	For $\phi\in \mathcal{R}ep(Q,E)$ we can identify $T_\phi\mathcal{R}ep(Q,E)\cong \mathcal{R}ep(Q,E)$. In analogy with the linear case, an hermitian metric on $\mathcal{R}ep(Q,E)$ is given by 
	$$
	\langle\phi,\psi\rangle=\sum_{\alpha\in Q_1}\text{Tr}(\phi_\alpha\circ \psi_\alpha^{*})
	$$
	for $\phi=(\phi_\alpha)_{\alpha\in Q_1},\psi=(\psi_\alpha)_{\alpha\in Q_1}\in\mathcal{R}ep(Q,E)$ and $\psi_\alpha^*:E_{h\alpha}\to M_\alpha\otimes E_{t\alpha}$ the smooth adjoint with respect to the hermitian metrics $H_i$ and $q_\alpha$. As expected, from what we know for the linear case,
	$$\omega=-2\text{Im}(\langle \_\ ,\_\rangle_{L^2})$$
	is a symplectic form on $\mathcal{R}ep(Q,E)$. The gauge group $\mathcal{G}$ acts on $\mathcal{R}ep(Q,E)$ as
	$$
	\begin{matrix}
		\mathcal{G}\times\mathcal{R}ep(Q,E) & \longrightarrow & \mathcal{R}ep(Q,E) \\
		((\Phi_i)_{i\in Q_0},(\varphi_\alpha)_{\alpha\in Q_1}) & \longmapsto & (\Phi_{h\alpha}\circ \varphi_\alpha\circ (\text{Id}_{M_\alpha}\otimes\Phi_{t\alpha})^{-1}).
	\end{matrix}
	$$
	As in Equation (\ref{defMomentMapLinearCase}), a moment map for the action of the gauge group just described is given by
	$$
	\begin{matrix}
		\mu: & \mathcal{R}ep(Q,E)&\longrightarrow & \text{Lie }\mathcal{G}^* \\
		&\varphi&\longmapsto&-\sqrt{-1} ([\varphi,\varphi^*]_i)_{i\in Q_0}.
	\end{matrix}
	$$
	Here, $[\varphi,\varphi^*]_i$ is as in Equation (\ref{commutatorMomentMap}) with the difference that we are now considering morphisms of vector bundles rather than linear maps. Also, notice that for all $\alpha\in Q_1$, $\varphi_\alpha$ and its adjoint $\varphi_\alpha^*$ can be seen as morphisms $\varphi_\alpha:E_{t\alpha}\to M_\alpha^*\otimes E_{h\alpha}$ and $\varphi_\alpha^*:M_\alpha^*\otimes E_{h\alpha}\to E_{t\alpha}$ so the composition $\varphi_\alpha^*\circ \varphi_\alpha$ in the commutators $[\varphi,\varphi^*]_i$ make sense. 
	\\
	\\
	The product $\mathcal{A}\times \mathcal{R}ep(Q,E)$ is a symplectic manifold for any form in the family
	$$
	\omega_\sigma=\sum_{i\in Q_0}\sigma_i \omega_i + \omega
	$$
	with $\sigma=(\sigma_i)_{i\in Q_0}\in \mathbb{R}^{|Q_0|}_{>0}$. 
	For a fixed symplectic form $\omega_\sigma$, a moment map for the action of the gauge group $\mathcal{G}$ on $\mathcal{A}\times \mathcal{R}ep(Q,E)$ is 
	\begin{align*}
		\mu_{\mathcal{A}\times \mathcal{R}ep(Q,E)}((\nabla_i)_{i\in Q_0},\varphi)&=\sum_{i\in Q_0}\sigma_i\mu_i (\nabla_i)+ \mu_\mathcal{S}(\varphi) \\
		&=(\sigma_i\Lambda F_i-\sqrt{-1}[\varphi,\varphi^*]_i)_{i\in Q_0}.
	\end{align*}
	In analogy with Theorem \ref{quiverVortexEqnsLinearCase}, one can study the set of solutions to the equations 
	\begin{equation}\label{infiniteDimensionalMomentMap}
		\mu_{\mathcal{A}\times \mathcal{R}ep(Q,E)}=c
	\end{equation}
	where $c\in \text{Lie }\mathcal{G}^*$ is fiberwise a fixed point for the coadjoint action of $\text{U}(d)$ on $\mathfrak{u}(d)^*$. This prompts the following definition:
	\begin{definition}\label{definitionQuiverVortexEquations}
		Let $\mathcal{E}=((\mathcal{E}_i)_{i\in Q_0},(\varphi_\alpha)_{\alpha\in Q_1})$ be a holomorphic $M$-twisted $Q$-bundle, $\sigma=(\sigma_i)_{i\in Q_0}\in \mathbb{R}^{|Q_0|}_{>0}$ and $\ \theta=(\theta_i)_{i\in Q_0}\in \mathbb{R}^{|Q_0|}$. A hermitian metric $H=\{H_i\}_{i\in Q_0}$ on $\mathcal{E}$ or a choice of a hermitian metric on each vector bundle $\mathcal{E}_i\to X$, satisfies the \emph{quiver $(\sigma,\theta)$-vortex equations} if 
		\begin{equation}\label{quiverVortexEquations}
			\sigma_i\sqrt{-1}\Lambda F_i+[\varphi,\varphi^*]_i=\theta_i\text{Id}_{E_i}
		\end{equation}
		for all vertex $i\in Q_0$. Here $F_i$ is the curvature of the Chern connection associated to the hermitian metric $H_i$ on the holomorphic vector bundle $\mathcal{E}_i\to X$. 
	\end{definition}
	We now discuss on this definition. Firstly, we remark that $-\sqrt{-1}\theta=(-\sqrt{-1}\theta_i\text{Id})_{i\in Q_0}$ defines a section of the Lie algebra of the gauge group $ \text{Lie }\mathcal{G}$. Moreover, the induced element in $\text{Lie }\mathcal{G}^*$  is, fiberwise, a fixed point for the coadjoint action of $\text{U}(d)$ on $\mathfrak{u}(d)^*$ and hence Equation (\ref{quiverVortexEquations}) is well-defined.  Secondly, we recall that the Chern connection on a hermitian and holomorphic vector bundle is the unique hermitian connection whose (0,1)-part coincides with the Dolbeault operator associated to the holomorphic structure. Thus, Definition \ref{definitionQuiverVortexEquations} picks the holomorphic solutions to Equation (\ref{infiniteDimensionalMomentMap}).
	\\
	\\
	The following theorem, which is due to Álvarez-Cónsul and García-Prada \cite[Theorem 3.1]{AlvarezConsulGarciaPradaHKCorrespondenceQuiversVortices}, describes the holomorphic solutions to the quiver vortex equations in terms of the stability conditions introduced in Definition \ref{defStableQuiverBundle}:
	\begin{theorem}[Hitchin-Kobayashi correspondence for (twisted) quiver bundles] A $M$-twisted holomorphic $Q$-bundle $\mathcal{E}=((\mathcal{E}_i)_{i\in Q_0},(\varphi_\alpha)_{\alpha\in Q_1})$ is $(\sigma,\theta)$-polystable if and only if it admits a hermitian metric $H$ satisfying the quiver $(\sigma,\theta)$-vortex equations.
	\end{theorem}
	\section{Applications}\label{applicationsSection}
	\subsection{Polystability of tensor products of quiver bundles}
	\begin{definition}
		Let $\mathcal{E}'=((\mathcal{E}_i')_{i\in Q_0'},(\varphi_{\alpha})_{\alpha\in Q_1'})$ be a $M$-twisted $Q'$-bundle and $\mathcal{E}''=((\mathcal{E}_j'')_{j\in Q_0''},$ $(\psi_{\beta})_{\beta\in Q_1''})$ be a $N$-twisted $Q''$-bundle. The \emph{tensor product} $Q'\otimes Q''$-bundle is the twisted quiver bundle denoted by $\mathcal{E}'\otimes \mathcal{E}''$ and given by the following data:
		\begin{itemize}
			\item To every vertex $(i,j)\in (Q'\otimes Q'')_0$ we assign the holomorphic vector bundle $\mathcal{E}_i'\otimes \mathcal{E}_j''$.
			\item To every edge $(\alpha,j)\in Q_0'\times Q_1''$ we assign the morphism $\delta_{(\alpha,j)}:=\varphi_{\alpha}\otimes \text{Id}:M_\alpha\otimes \mathcal{E}_{t\alpha}'\otimes \mathcal{E}_j''\to \mathcal{E}_{h\alpha}'\otimes \mathcal{E}_j''$ and to every edge $(i,\beta)$ we assign the morphism $\delta_{(i,\beta)}:=\text{Id}\otimes \psi_\beta:\mathcal{E}_i'\otimes N_\beta\otimes  \mathcal{E}_{t\alpha}''\to \mathcal{E}_i'\otimes \mathcal{E}_{h\alpha}''$.
		\end{itemize}
	\end{definition}
	\begin{remark}
		Since there is an equivalence of categories between locally free $\mathcal{O}_X$-modules and holomorphic vector bundles over $X$, the tensor product defined above is just an instance of a tensor product of twisted representations as defined in Section \ref{tensorProductsTwistedRepresentations}.
	\end{remark}
	\begin{example}\label{extendedHomQuiverExample}
		Let  $\mathcal{E}'=((\mathcal{E}_i')_{i\in Q_0},(\varphi_{\alpha})_{\alpha\in Q_1})$ and $\mathcal{E}''=((\mathcal{E}_j'')_{j\in Q_0},$ $(\psi_{\beta})_{\beta\in Q_1})$ be $Q$-bundles without twisting. We define the \emph{dual quiver bundle} of $\mathcal{E}'$ to be the $Q^{op}$-bundle given by $(\mathcal{E}')^*=((\mathcal{E}_i'^*)_{i\in Q_0},(-\varphi_\alpha^*)_{\alpha\in Q_1})$ with $Q^{op}$ the quiver introduced in Example \ref{oppositeQuiverExample}.  Here $\varphi_\alpha^*:\mathcal{E}_{h\alpha}'^*\to \mathcal{E}_{t\alpha}'^*$  is the morphism induced by $\varphi_\alpha:\mathcal{E}'_{t\alpha}\to \mathcal{E}'_{h\alpha}$ by duality and should not be confused with the adjoint morphism with respect to a choice of hermitian metrics.
		\\
		\\
		The $Q^{op}\otimes Q$-bundle $(\mathcal{E}')^*\otimes \mathcal{E}''$ is the so-called \emph{extended} \text{Hom}-\emph{quiver bundle} and was introduced by Gothen and Nozad in the context of holomorphic chains ($Q$-bundles with $Q$ of type $A_m$) over a compact Riemann surface $X$ \cite[Definition 3.1]{GothenNozad}. Another instance of this quiver bundle, on a different guise though, first appeared on the work of Bradlow, García-Prada and Gothen on holomorphic triples ($Q$-bundles with $Q$ of type $A_2$) \cite[Section 4.2]{BGGTriplesPaper}.
		\\
		\\
		More concretely, the quiver bundle
		
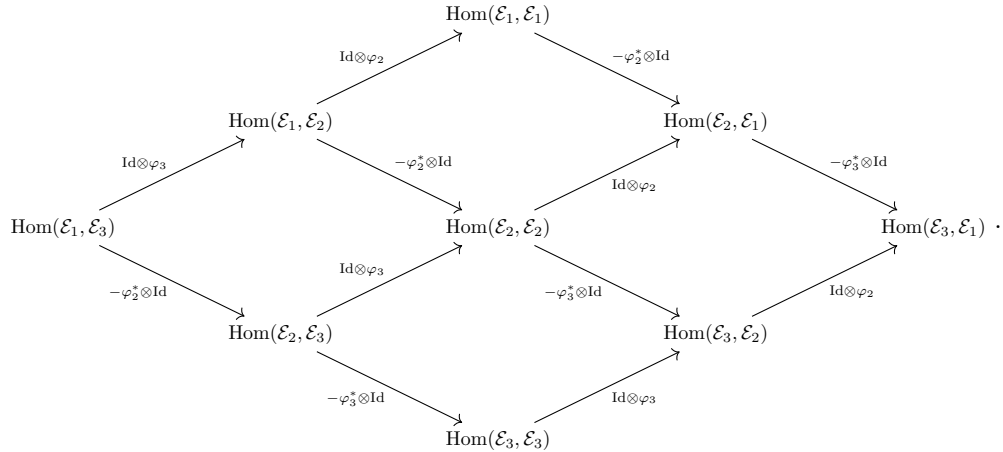
\begin{center}
			\begin{tikzcd}[scale cd=0.75, column sep =large, row sep =large]
				& & \text{Hom}(\mathcal{E}_1,\mathcal{E}_1) \arrow[dr,,"-\varphi_2^*\otimes \text{Id}"]&&  \\
				&\text{Hom}(\mathcal{E}_1,\mathcal{E}_2)\arrow[ur,"\text{Id}\otimes \varphi_2"] \arrow[dr,"-\varphi_2^*\otimes \text{Id}"]& &\text{Hom}(\mathcal{E}_2,\mathcal{E}_1) \arrow[dr,"-\varphi_3^*\otimes \text{Id}"]&   \\
				\text{Hom}(\mathcal{E}_1,\mathcal{E}_3) \arrow[ur,"\text{Id}\otimes \varphi_3"] \arrow[dr,"-\varphi_2^*\otimes \text{Id}",']&  & 	\text{Hom}(\mathcal{E}_2,\mathcal{E}_2) \arrow[ur,"\text{Id}\otimes \varphi_2",'] \arrow[dr,"-\varphi_3^*\otimes \text{Id}",']&  & \text{Hom}(\mathcal{E}_3,\mathcal{E}_1)\\
				&\text{Hom}(\mathcal{E}_2,\mathcal{E}_3)\arrow[ur,"\text{Id}\otimes \varphi_3"]  \arrow[dr,"-\varphi_3^*\otimes \text{Id}",']& &\text{Hom}(\mathcal{E}_3,\mathcal{E}_2) \arrow[ur,"\text{Id}\otimes \varphi_2",']&  \\
				& & \text{Hom}(\mathcal{E}_3,\mathcal{E}_3) \arrow[ur,"\text{Id}\otimes \varphi_3",'] &&
			\end{tikzcd}.
		\end{center}
		is, for instance, the extended Hom-quiver bundle, $\mathcal{E}^*\otimes \mathcal{E}$, for the holomorphic chain 
		\begin{center}
			\begin{tikzcd}[scale cd=0.9, column sep =large, row sep =large]
				\mathcal{E}:\ \mathcal{E}_3 \arrow[r,"\varphi_3"]& \mathcal{E}_2 \arrow[r,"\varphi_2"] &\mathcal{E}_1.
			\end{tikzcd}
		\end{center}
	\end{example}
	\begin{theorem}\label{polystabilityTensorProductQuiverBundles}
		Let $\mathcal{E}',\mathcal{E}''$ be $(\sigma',\theta')$ and $(\sigma'',\theta'')$-polystable twisted quiver bundles respectively. Then, $\mathcal{E}'\otimes \mathcal{E}''$ is $(\sigma,\theta)$-polystable for $\sigma=(\sigma_i'\sigma_j'')_{(i,j)\in (Q'\otimes Q'')_0}$ and $\theta=(\theta_i'\sigma_j''+\theta_j''\sigma_j')_{(i,j)\in (Q'\otimes Q'')_0}$.
	\end{theorem}
	\begin{proof}
		For the sake of readability we give a proof of this theorem for the untwisted case. The proof for the twisted counterpart follows the same steps just taking into account that the smooth adjoints of the morphisms $\varphi_{\alpha}:M_\alpha\otimes \mathcal{E}_{t\alpha}\to \mathcal{E}_{h\alpha}$ should be computed with respect to the hermitian metrics $H_{t\alpha}q_\alpha$ ($q_\alpha$ being the already fixed hermitian metric on the twisting vector bundle $M_\alpha$) and $H_{h\alpha}$ on the bundles $M_\alpha\otimes\mathcal{E}_{t\alpha}\to X$ and $\mathcal{E}_{h\alpha}\to X$ respectively.  
		\\
		\\
		First, let us remark that at each vertex $(i,j)\in (Q'\otimes Q'')_0$ the tensor quiver bundle $\mathcal{E}'\otimes \mathcal{E}''$ takes the form:
		
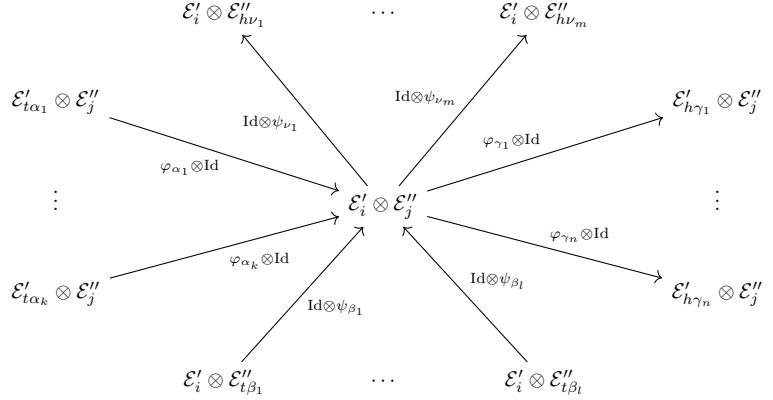
\begin{center}
			\begin{tikzcd}[scale cd=0.8, column sep =normal, row sep =normal]		
				& \mathcal{E}_i'\otimes \mathcal{E}''_{h\nu_1}& \cdots&\mathcal{E}_i'\otimes \mathcal{E}''_{h\nu_m}&  \\
				\mathcal{E}_{t\alpha_1}'\otimes \mathcal{E}_j''\arrow[to=3-3,"\varphi_{\alpha_1}\otimes \text{Id}",']&& && \mathcal{E}_{h\gamma_1}'\otimes \mathcal{E}_j''  \\
				\vdots &  & \mathcal{E}_i'\otimes \mathcal{E}_j'' \arrow[to=1-2,"\text{Id}\otimes \psi_{\nu_1}"]\arrow[to=1-4,"\text{Id}\otimes \psi_{\nu_m}"]\arrow[to=2-5,"\varphi_{\gamma_1}\otimes \text{Id}"]\arrow[to=4-5,"\varphi_{\gamma_n}\otimes \text{Id}"]&  &\vdots\\
				\mathcal{E}_{t\alpha_k}'\otimes \mathcal{E}_j''\arrow[to=3-3,"\varphi_{\alpha_k}\otimes \text{Id}",']&& &&   \mathcal{E}_{h\gamma_n}'\otimes \mathcal{E}_j''\\
				&\mathcal{E}_i'\otimes \mathcal{E}''_{t\beta_1}\arrow[to=3-3,"\text{Id}\otimes \psi_{\beta_1}",']& \cdots &\mathcal{E}_i'\otimes \mathcal{E}''_{t\beta_l}\arrow[to=3-3,"\text{Id}\otimes \psi_{\beta_l}",']&
			\end{tikzcd}.
		\end{center}
		Here $\alpha_1,\ldots,\alpha_k,\gamma_1,\ldots,\gamma_n\in Q'_1$, $\beta_1,\ldots,\beta_l,\nu_1,\ldots,\nu_m\in Q_1''$ are such that $h\alpha_1=\cdots=h\alpha_k=t\gamma_1=\cdots=t\gamma_n=i$ and  $h\beta_1=\cdots=h\beta_l=t\nu_1=\cdots=t\nu_m=j$.
		\\
		\\
		Since $\mathcal{E}'$ and $\mathcal{E}''$ are polystable, there exist hermitian metrics $H_i'$ and $H_j''$ on $\mathcal{E}_i'\to X$ and $\mathcal{E}_j''\to X$ respectively such that the corresponding quiver vortex equations hold. We have a natural induced hermitian metric on $\mathcal{E}_i'\otimes \mathcal{E}_j''\to X$ given by the product of the hermitian metrics $H_i'$ and $H_j''$. The corresponding Chern connection is given by 
		$$
		\nabla_{\mathcal{E}_i'\otimes \mathcal{E}_j''}=\nabla_{\mathcal{E}_i'}\otimes \text{Id} + \text{Id}\otimes \nabla_{\mathcal{E}_j''},
		$$
		where $\nabla_{\mathcal{E}_i'}$ and $\nabla_{\mathcal{E}_j''}$ are the Chern connections on $\mathcal{E}_i'$ and $\mathcal{E}_j''$ respectively, and the curvature of this connection  is therefore given by 
		$$
		F_{\mathcal{E}_i'\otimes \mathcal{E}_j''}=F_{\mathcal{E}_i'}\otimes \text{Id}+\text{Id}\otimes F_{\mathcal{E}_j''}.
		$$
		The adjoint morphism, $(\varphi_{\alpha_1}\otimes \text{Id})^*:\mathcal{E}_i'\otimes \mathcal{E}_j''\to \mathcal{E}_{t\alpha_1}'\otimes \mathcal{E}_j''$ , to $\varphi_{\alpha_1}\otimes \text{Id}$ is $\varphi_{\alpha_1}^*\otimes \text{Id}$ and similarly for the others.  With this in mind we now write the quiver vortex equations for the tensor quiver bundle $\mathcal{E}'\otimes \mathcal{E}''$. 
		\\
		\\
		Let $s_i',s_j''$ be smooth sections of the vector bundles $\mathcal{E}_i'\to X$ and $\mathcal{E}_j''\to X$ respectively.  Then, 
		\begin{align*}
			\sqrt{-1}\Lambda F_{\mathcal{E}_i'\otimes \mathcal{E}_j''}(s_i'\otimes s_j'')=\sqrt{-1}\Lambda F_{\mathcal{E}_i'}s_i'\otimes s_j'' +s_i'\otimes \sqrt{-1}\Lambda F_{\mathcal{E}_j''}s_j''.
		\end{align*}
		On the other hand, 
		$$
		\bigg (\sum_{\substack{(i,j)=h(\alpha,j) \\ (i,j)=h(i,\beta)}}\delta_{(\alpha,j)}\delta_{(\alpha,j)}^*+\delta_{(i,\beta)}\delta_{(i,\beta)}^*\bigg)(s_i'\otimes s_j'')=\sum_{h\alpha=i}\varphi_\alpha\varphi_{\alpha}^*(s_i')\otimes s_j''+s_i'\otimes\sum_{h\beta=j}\psi_\beta\psi_\beta^*(s_j'')
		$$
		and 
		$$
		\bigg(\sum_{\substack{(i,j)=t(\alpha,j)\\(i,j)=t(i,\beta)}}\delta_{(\alpha,j)}^*\delta_{(\alpha,j)}+\delta^*_{(i,\beta)}\delta_{(i,\beta)}\bigg)(s_i'\otimes s_j'')=\sum_{t\alpha=i}\varphi_\alpha^*\varphi_\alpha(s_i')\otimes s_j''+s_i'\otimes\sum_{t\beta=j}\psi_\beta^*\psi_\beta(s_j'').
		$$
		Thus, by combining these equalities, rearranging terms and using the fact that the quiver bundles $\mathcal{E}'$ and $\mathcal{E}''$ satisfy the quiver vortex equations we obtain
		\begin{align*}
			&(\sigma_i'\sigma_j''\sqrt{-1}\Lambda F_{\mathcal{E}_i'\otimes \mathcal{E}_j''}+[\delta,\delta^*]_{(i,j)} )(s_i'\otimes s_j'')\\	
			&= (\sigma_i'\sqrt{-1}\Lambda F_{\mathcal{E}_i'}+[\varphi,\varphi^*]_i)(s_i')\otimes \sigma_j''s_j'' +\sigma_i's_i'\otimes (\sigma_j''\sqrt{-1}\Lambda F_{\mathcal{E}_j''}+[\psi,\psi^*]_j)(s_j'')\\
			&= (\theta_i'\sigma_j''+\sigma_i'\theta_j'')(s_i'\otimes s_j'')
		\end{align*}
		and the desired conclusion follows. 
	\end{proof}
	
	We have pointed out several times that when $X=\text{Spec}(\mathbb{C})$, $Q$-bundles are just quiver representations in the category of finite dimensional vector spaces over $\mathbb{C}$. In this setting, we have no curvature term so the quiver vortex equations are precisely those described in Equation (\ref{quiverVortexEquationsForAPoint}). The following is an easy corollary from our theorem. We have been aware that Das et al. have arrived to the same result in a recent work \cite[Lemma 2.8]{DasEtAlTensors}.
	
	\begin{corollary}\label{polystabilityTensorProductQuiverRepresentations}
		Let $\varphi=(\varphi_\alpha)_{\alpha\in Q'}$ and $\psi=(\psi_\beta)_{\beta\in Q''}$ be $\theta'$, $\theta''$-polystable quiver representations respectively. Then, their tensor product {\normalfont $\varphi\otimes \psi:=(\varphi_\alpha\otimes \text{Id})_{(\alpha,j)}\sqcup(\text{Id}\otimes \psi_\beta)_{(i,\beta)}$}, which is a representation of the tensor quiver $Q'\otimes Q''$, is a $\theta:=(\theta_i'+\theta_j'')_{(i,j)\in (Q'\otimes Q'')_0}$-polystable quiver representation. 
	\end{corollary}
	\subsection{Interlude: Collapsing and expanding quiver bundles}\label{operationsOnQuiversAndStability}
	We can produce, given a fixed quiver $Q$, a new one, say $Q'$, by performing one of the following operations: \emph{deleting, collapsing and clonning vertices or edges}. For instance, Figure \ref{collapsingVertices} and Figure \ref{collapsingEdges} show how do vertex and edge collapsing work respectively. 
	\begin{remark}
		When deleting (cloning) a vertex, all the edges connected to this should also be deleted (cloned). 
	\end{remark}
	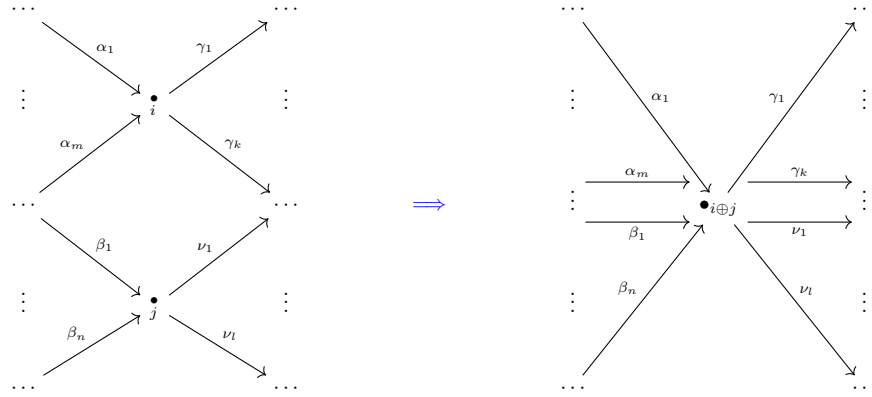
\begin{figure}[H]
		\begin{center}
			\begin{tikzcd}[scale cd=0.8, column sep =large, row sep =normal]
				\cdots\arrow[to=2-2,"\alpha_1"]&&\cdots&&\cdots \arrow[to=3-6,"\alpha_1"]&&  \cdots\\
				\vdots&{}^\bullet_i\arrow[to=1-3,"\gamma_1"]\arrow[to=3-3,"\gamma_k"]&\vdots&&\vdots&&\vdots  \\
				\cdots\arrow[to=2-2,"\alpha_m"]\arrow[to=4-2,"\beta_1"]&&\cdots&\textcolor{blue}{\Longrightarrow}&\vdots\arrow[r,shift left=2ex,"\alpha_m"]\arrow[r, shift right=1.5ex,"\beta_1",']&\bullet_{i\oplus j}\arrow[to=1-7,"\gamma_1"]\arrow[to=5-7,"\nu_l"]\arrow[r,shift left=2ex,"\gamma_k"]\arrow[r, shift right=1.5ex,"\nu_1",']&\vdots \\
				\vdots&{}^\bullet_j\arrow[to=3-3,"\nu_1"]\arrow[to=5-3,"\nu_l"]&\vdots&&\vdots&&\vdots \\
				\cdots \arrow[to=4-2,"\beta_n"]&&\cdots &&\cdots \arrow[to=3-6,"\beta_n"]&& \cdots 
			\end{tikzcd}
		\end{center}
		\caption{\label{collapsingVertices} Collapsing vertices}
	\end{figure}
	Now, we briefly discuss the effect that the aforementioned operations have on the corresponding path algebras. Following the notation from Section \ref{representationTheoryOfTensorQuivers}, we let $\{\mathscr{M}_\alpha\}_{\alpha\in Q_1}$ be a collection of twisting $\mathcal{O}_X$-modules and we note that when performing a deleting or collapsing operation we have an injection of path algebras $\mathcal{T}_{\mathscr{M}'}\mathcal{A}_{Q'}\hookrightarrow \mathcal{T}_\mathscr{M}\mathcal{A}_Q$
	for $\mathscr{M}=\bigoplus_{\alpha\in Q_1}\mathscr{M}_\alpha$. It is worth noting that if $Q'$ is obtained from collapsing the edges $\alpha_1,\ldots,\alpha_k$, then we set the twisting $\mathcal{O}_X$-module $\mathscr{M}'_{\alpha_1+\cdots+\alpha_k}$ to be the direct sum $\mathscr{M}_{\alpha_1}\oplus\cdots\oplus\mathscr{M}_{\alpha_k}$. When cloning either a vertex or an edge, there is a canonical choice of twisting $\mathcal{O}_X$-modules for each of the cloned edges and contrary to what happened previously, there will be a surjection of path algebras $\mathcal{T}_{\mathscr{M}'}\mathcal{A}_{Q'}\twoheadrightarrow \mathcal{T}_{\mathscr{M}}\mathcal{A}_Q$.
	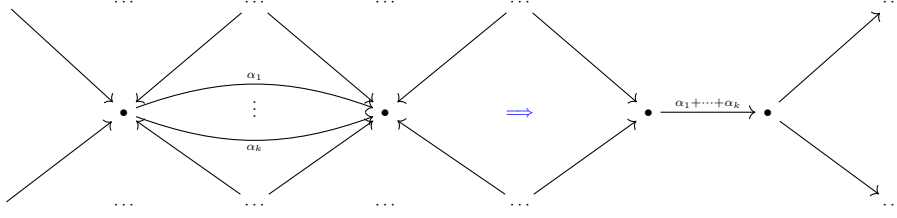
\begin{figure}[H]
		\begin{center}
			\begin{tikzcd}[scale cd=0.65, column sep =large, row sep =large]
				\arrow[to=2-2]&\cdots&\cdots\arrow[to=2-2]\arrow[to=2-4]&\cdots&\cdots\arrow[to=2-4]\arrow[to=2-6]&
				&&\cdots\\
				&\bullet \arrow[to=2-4,',bend left=20,"\alpha_1"']\arrow[to=2-4,',bend right=20,"\alpha_k"]&\vdots&\bullet&\textcolor{blue}{\Longrightarrow}&\bullet\arrow[r,"\alpha_1+\cdots+\alpha_k"]&\bullet\arrow[to=3-8]\arrow[to=1-8]&\\
				\arrow[to=2-2]&\cdots&\cdots\arrow[to=2-2]\arrow[to=2-4]&\cdots&\cdots\arrow[to=2-4]\arrow[to=2-6]&&&\cdots
			\end{tikzcd}
		\end{center}
		\caption{\label{collapsingEdges} Collapsing edges}
	\end{figure}
	From the preceeding discussion it follows that:
	\begin{proposition}\label{restrictionOfScalarsAndOperationsOnQuivers}
		Restriction of scalars induced by the morphisms $\mathcal{T}_{\mathscr{M}'}\mathcal{A}_{Q'}\to \mathcal{T}_{\mathscr{M}}\mathcal{A}_Q$ described above gives a functor {\normalfont $\mathcal{T}_\mathscr{M}\mathcal{A}_Q\text{-mod}\to \mathcal{T}_{\mathscr{M}'}\mathcal{A}_{Q'}\text{-mod}$}. 
		
	\end{proposition}
	\begin{example}
		Following Example \ref{exampleHiggsBundles}, for $i=1,2$, let $(\mathcal{E}_i,\phi_i:\mathcal{E}_i\otimes K_X^*\to \mathcal{E}_i)$ be Higgs bundles over $X$. The quiver-theoretic tensor product of these two Higgs bundles is a twisted representation of the quiver $Q$, given by the tensor product of the Jordan quiver (see Figure \ref{JordanQuiver}) with itself, depicted in Figure \ref{eightQuiver}. In this case, the twisting $\mathcal{O}_X$-module, for each edge of $Q$, will be $K_X^*$.  If we collapse the edges $\alpha,\beta\in Q_1$, then the twisted representation of the quiver $Q'=Q_J$ obtained by restriction of scalars will be the Higgs bundle $(\mathcal{E}_1\otimes \mathcal{E}_2,\phi_1\otimes \text{Id}+\text{Id}\otimes \phi_2)$. This is, in fact, how the tensor product is defined in the category of Higgs bundles over $X$ (see, for instance, \cite[p.14]{Simpson}). 
	\end{example}
	\vspace{-7mm}
	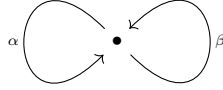
\begin{figure}[H]
		\begin{center}
			\begin{tikzcd}[scale cd=0.8, column sep =normal, row sep =normal]
				\bullet \arrow[loop,in=45,out=315,distance= 20mm,"\beta",']  \arrow[loop,in=225,out=135,distance= 20mm,"\alpha",']
			\end{tikzcd}
		\end{center}
		\vspace{-7mm}
		\caption{\label{eightQuiver} The tensor quiver $Q_J\otimes Q_J$}
	\end{figure}
	Using the same approach as in the proof of Theorem \ref{polystabilityTensorProductQuiverBundles}, we can show:
	\begin{proposition}\label{effectOfOperationsInPolystability}
		Let $\mathcal{E}=((\mathcal{E}_i)_{i\in Q_0},(\varphi_\alpha)_{\alpha\in Q_1})$ be a $(\sigma,\theta)$-polystable twisted Q-bundle. The following claims hold:
		\begin{enumerate}
			\item If $\mathcal{E}'$ is the twisted $Q'$-bundle obtained from Proposition \ref{restrictionOfScalarsAndOperationsOnQuivers} after collapsing any number of vertices of $Q$ for which the corresponding component of the stability parameter is equal, then $\mathcal{E}'$ is $(\sigma',\theta')$-polystable for $\sigma'=(\sigma_i)_{i\in Q_0'}$ and $\theta'=(\theta_i)_{i\in Q'_0}$.
			\item If $\mathcal{E}'$ is the twisted $Q'$-bundle obtained from Proposition \ref{restrictionOfScalarsAndOperationsOnQuivers} after collapsing the edges $\alpha_1,\ldots,\alpha_k\in Q_1$ and if {\normalfont
				$$\sum_{n,m=1 }^k\varphi_{\alpha_n}\varphi_{\alpha_m}^*=\tau\text{Id}_{\mathcal{E}_{h\alpha_n}}, \sum_{n,m=1}^k\varphi_{\alpha_n}^*\varphi_{\alpha_m}=\tau\text{Id}_{\mathcal{E}_{t\alpha_n}},$$} 
			for some $\tau\in \mathbb{R}$, then $\mathcal{E}'$ is $(\sigma,\theta')$-polystable, for 
			$$
			\theta'_i=\begin{cases}
				\theta_i-\tau & \text{if }i=h\alpha_1=\cdots=h\alpha_k,\\
				\theta_i+\tau & \text{if }i=t\alpha_1=\cdots=t\alpha_k, \\
				\theta_i & \text{otherwise}.
			\end{cases}
			$$ 
			\item If $\mathcal{E}'$ is the twisted $Q'$-bundle obtained from Proposition \ref{restrictionOfScalarsAndOperationsOnQuivers} after cloning any number of vertices of $Q$, then $\mathcal{E}'$ is polystable for the obvious choice of stability parameters.
			\item If $\mathcal{E}'$ is the twisted $Q'$-bundle obtained from Proposition \ref{restrictionOfScalarsAndOperationsOnQuivers} after cloning (deleting) the edge $\alpha\in Q_1$ and if {\normalfont
				$$
				\varphi_\alpha\varphi^*_\alpha=\tau\text{Id}_{\mathcal{E}_{h\alpha}}, \varphi_\alpha^*\varphi_\alpha=\tau\text{Id}_{\mathcal{E}_{t\alpha}}
				$$} 
			for some $\tau\in \mathbb{R}$,
			then $\mathcal{E}'$ is $(\sigma,\theta')$-polystable for 
			, for 
			$$
			\theta'_i=\begin{cases}
				\theta_i-(+)\tau & \text{if }i=h\alpha,\\
				\theta_i+(-)\tau & \text{if }i=t\alpha, \\
				\theta_i & \text{otherwise}.
			\end{cases}
			$$ 
		\end{enumerate}
	\end{proposition}
	\begin{example}\label{deformationComplexExample}
		Following Example \ref{extendedHomQuiverExample}, starting from a $\sigma=(\sigma_i)_{i\in Q_0}$-polystable quiver bundle $\mathcal{E}$, the extended Hom-quiver bundle $\mathcal{E}^*\otimes \mathcal{E}$ will be $(\sigma_j-\sigma_i)_{(i,j)\in Q_0\times Q_0}$-polystable by Theorem \ref{polystabilityTensorProductQuiverBundles}. If we collapse all the vertices of the form $(i,i)\in (Q^{op}\otimes Q)_0$, Proposition \ref{restrictionOfScalarsAndOperationsOnQuivers} and Proposition \ref{effectOfOperationsInPolystability} will give us a polystable quiver bundle which is obtained from $\mathcal{E}^*\otimes \mathcal{E}$ by collapsing all the holomorphic vector bundles of the form $\text{Hom}(\mathcal{E}_i,\mathcal{E}_i)$ as shown next for quivers of type $A_3$ (compare with the quiver bundle in Example \ref{extendedHomQuiverExample}). 
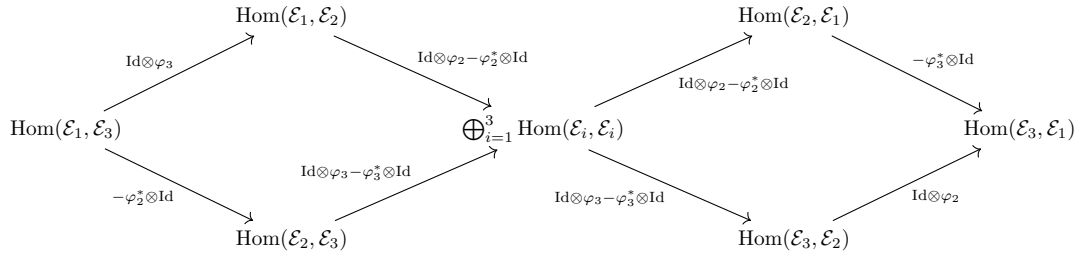
		\begin{center}
			\begin{tikzcd}[scale cd=0.8, column sep =large, row sep =large]
				&\text{Hom}(\mathcal{E}_1,\mathcal{E}_2) \arrow[dr,"\text{Id}\otimes \varphi_2-\varphi_2^*\otimes \text{Id}"]& &\text{Hom}(\mathcal{E}_2,\mathcal{E}_1) \arrow[dr,"-\varphi_3^*\otimes \text{Id}"]&   \\
				\text{Hom}(\mathcal{E}_1,\mathcal{E}_3) \arrow[ur,"\text{Id}\otimes \varphi_3"] \arrow[dr,"-\varphi_2^*\otimes \text{Id}",']&  & 	\bigoplus_{i=1}^3\text{Hom}(\mathcal{E}_i,\mathcal{E}_i) \arrow[ur,"\text{Id}\otimes \varphi_2-\varphi_2^*\otimes \text{Id}",'] \arrow[dr,"\text{Id}\otimes \varphi_3-\varphi_3^*\otimes \text{Id}",']&  & \text{Hom}(\mathcal{E}_3,\mathcal{E}_1)\\
				&\text{Hom}(\mathcal{E}_2,\mathcal{E}_3)\arrow[ur,"\text{Id}\otimes \varphi_3-\varphi_3^*\otimes \text{Id}"] & &\text{Hom}(\mathcal{E}_3,\mathcal{E}_2) \arrow[ur,"\text{Id}\otimes \varphi_2",']& 
			\end{tikzcd}
		\end{center}
		For $Q$ of type $A_m$, Gothen and Nozad \cite{GothenNozad} used this collapsed quiver bundle to give an algebraic proof of an earlier result by Álvarez-Cónsul et al. relevant in the deformation theory of the moduli space of holomorphic chains over a compact Riemann surface \cite{AlvarezConsulGarciaPradaSchmitt}.  
		\\
		\\
		If we furthermore collapse those vector bundles associated to the outgoing edges of the vertices we have just collapsed and then we discard the rest we obtain the two-term complex:
		$$
		C^\bullet(\mathcal{E}):\bigoplus_{i\in Q_0}\text{Hom}(\mathcal{E}_i,\mathcal{E}_i)\longrightarrow \bigoplus_{\alpha\in Q_1}\text{Hom}(\mathcal{E}_{t\alpha},\mathcal{E}_{h\alpha}).
		$$ 
		This is the so-called \emph{deformation complex} of the quiver bundle $\mathcal{E}$. When $\mathcal{E}$ is a semistable holomorphic chain, this complex encodes local information, on a neighborhood of the equivalence class representing $\mathcal{E}$, of the moduli space of semistable holomorphic chains, see for instance \cite[Theorem 3.8]{BGGTriplesPaper} and \cite[Theorem 3.8]{AlvarezConsulGarciaPradaSchmitt}. 
		This is a further  compelling feature of the quiver bundle $\mathcal{E}^*\otimes \mathcal{E}$ and it would be interesting to understand the extra local information the rest of this quiver bundle is giving of the corresponding moduli space. 
	\end{example}
	
	\subsection{Segre embedding for quiver representations}\label{segreEmbeddingForQuiverRepresentations}
	The results in Section \ref{representationTheoryOfTensorQuivers}, Theorem \ref{polystabilityTensorProductQuiverBundles} and its subsequent corollary tell us that tensor products of polystable quiver bundles give rise to polystable representations of the tensor quiver. The polystability condition, as we already have seen, basically means that a set of equations needs to be satisfied. There could exist, however, polystable representations of the tensor quiver that do not come from tensor products of polystable quiver bundles. It is therefore natural to ask if the set of polystable representations of the tensor quiver that come from tensor products of polystable quiver representations has any sort of analytic or algebraic structure. One could think of the question we are asking as a quiver version of that the classical Segre embedding answers: if $V,W$ are finite-dimensional complex vector spaces, then decomposable tensors of $V\otimes W$ can be described by a closed subvariety of the projective space $\mathbb{P}(V\otimes W)$ (see, for instance \cite[p.25]{HarrisAG}).
	\\
	\\
	In the rest of this section we deal with this problem for the linear case, that is, for $X=\text{Spec}(\mathbb{C})$. Let $Q',Q''$ be quivers, $d'\in \mathbb{N}^{|Q_0'|}, \ d''\in \mathbb{N}^{|Q_0''|}$  be indivisible dimension vectors and $\theta'\in \mathbb{Z}^{|Q_0'|},\theta''\in\mathbb{Z}^{|Q_0''|}$ stability parameters for which the symplectic quotients $\mu^{-1}(\theta')/\text{U}(d')$ and $\mu^{-1}(\theta'')/\text{U}(d'')$ are smooth manifolds. Here we have denoted $\mu^{-1}(\theta)$, for the sake of simplicity, the subset of $\text{Rep}(Q,d)$ given by the solutions to the equations $[\varphi,\varphi^*]_i=\theta_i\text{Id}$ for all $i\in Q_0$. By Theorem \ref{quiverVortexEqnsLinearCase} we know that this subset corresponds to the $\theta$-polystable quiver representations of $Q$. Now, Corollary \ref{polystabilityTensorProductQuiverRepresentations} implies that there is a well defined map 
	\begin{equation}\label{tensorProductMap}
		\begin{matrix}
			\mu^{-1}(\theta')\times \mu^{-1}(\theta'') & \longrightarrow & \mu^{-1}(\theta) \\
			(\varphi,\psi) &\longmapsto  & \varphi\otimes \psi. 
		\end{matrix}
	\end{equation}
	This map is equivariant for the group morphism $$\begin{matrix}
		\text{U}(d')\times \text{U}(d'')& \longrightarrow&\text{U}(d) \\
		((M_i)_{i\in Q_0'},(N_j)_{j\in Q_0''}) & \longmapsto & (M_i\otimes N_j)_{(i,j)\in (Q'\otimes Q'')_0}
	\end{matrix} $$
	for $d=(d_i'd_j'')_{(i,j)\in (Q'\otimes Q'')_0}$. Hence, the map in Equation (\ref{tensorProductMap}) factors through a map 
	$$
	\begin{matrix}
		F: & \mu^{-1}(\theta')/\text{U}(d')\times \mu^{-1}(\theta'')/\text{U}(d'') & \longrightarrow &\mu^{-1}(\theta)/\text{U}(d).
	\end{matrix}
	$$
	The indivisibility of $d$ and a generic choice of stability parameters $\theta',\theta''$ imply that the latter quotient above can be assumed to be a manifold (see Remark \ref{indivisibilityDimensionVector}). 
	We claim that $F$ is an embedding. We will check this in two steps: first we see that $F$ is an immersion and then we check its injectivity. 
	\\
	\\
	The preimage theorem tells us that the tangent space at $\varphi=(\varphi_{\alpha})_{\alpha\in Q_1}\in \mu^{-1}(\theta)$ is $T_\varphi\mu^{-1}(\theta)=\ker(\text{d}_\varphi\mu)$. A direct computation gives that $A=(A_\alpha)_{\alpha\in Q_1}\in \text{ker}(\text{d}_\varphi\mu)$ if and only if 
	$$
	\sum_{h\alpha = i}\varphi_\alpha A_\alpha^*+A_{\alpha}\varphi_\alpha^*=\sum_{t\alpha=i}\varphi_{\alpha}^*A_\alpha+A_\alpha^*\varphi_\alpha
	$$
	for all $i\in Q_0$. Also we have the following complex of vector spaces
	\begin{center}
		\begin{tikzcd}[column sep =large, row sep =large]
			0\arrow[r]&\mathfrak{u}(d)\arrow[r,"d_0"]& T_{\varphi}\text{Rep}(Q,d) \arrow[r,"d_1"] &\mathfrak{u}(d)\arrow[r]&0
		\end{tikzcd}
	\end{center}
	where $d_0$ is the infinitesimal action described in Equation (\ref{infinitesimalActionLieAlgebraUnitaryGroups}) and $d_1=\text{d}_\varphi\mu$. Therefore, 
	$$
	T_{[\varphi]}\mu^{-1}(\theta)/\text{U}(d)=\ker(d_1)/\text{Im}(d_0)\cong\ker(d_1)\cap\ker(d_0^*)  
	$$
	where $d_0^*:T_{\varphi}\text{Rep}(Q,d)\to \mathfrak{u}(d)$ is the adjoint to $d_0$ which can be computed to be 
	$$
	d_0^*(\psi)=[\psi,\varphi^*]=(\sum_{h\alpha=i}\psi_\alpha\varphi_\alpha^*-\sum_{t\alpha=i}\varphi_{\alpha}^*\psi_\alpha)_{i\in Q_0}
	$$
	for $\psi=(\psi_\alpha)_{\alpha\in Q_1}\in T_\varphi\text{Rep}(Q,d)$. 
	\\
	\\
	To see that $F$ is an immersion we first note that the tensorization map $\text{Rep}(Q',d')\times \text{Rep}(Q'',d'')\to \text{Rep}(Q'\otimes Q'',d)$, which is the same as that in Equation (\ref{tensorProductMap}), is a linear map so its derivative is itself. This map and hence its derivative are clearly injective. On the other hand, our preceeding discussion identifies the tangent space $T_{[\varphi]}\text{Rep}(Q,d)/\text{U}(d)$ with a subspace of $T_{\varphi}\text{Rep}(Q,d)\cong \text{Rep}(Q,d)$. Thus, to show the claimed statement about $F$, it is enough to show that if $A\in\ker(d_1')\cap \ker(d_0'^*)$, $B\in\ker(d_1'')\cap \ker(d_0''^*)$ then $A\otimes B\in \ker(d_0)\cap\ker(d_1^*)$. A straightforward computation shows that this is indeed true. 
	\\
	\\
	Finally, we deal with the injectivity of the map $F$.  It is enough to show that if $\varphi_1\otimes \psi_1$ and $\varphi_2\otimes \psi_2$ are in the same $\text{U}(d)$-orbit then $\varphi_1$ lies in the same $\text{U}(d')$-orbit as $\varphi_2$ and similarly for $\psi_1$ with $\psi_2$. Here $\varphi_k\in \mu^{-1}(\theta')$ and $\psi_k\in \mu^{-1}(\theta'')$ for $k=1,2$. Note that for a fixed $j\in Q_0''$, the tuple
	$$\varphi_{k,j}:=(\varphi_{k,\alpha}\otimes \text{Id}_j=\varphi_\alpha^{\oplus d''_j}:(\mathbb{C}^{t\alpha})^{\oplus d_j''}\to (\mathbb{C}^{h\alpha})^{\oplus d_j''})_{\alpha\in Q_1'}$$
	defines a representation in $\text{Rep}(Q',d_j''d)$. In fact, this representation is $\theta'$-polystable as it is isomorphic to the direct sum of $d_j''$ copies of $\varphi$. Furthermore, note that $\varphi_{1,j}$ is in the same $\text{U}(d_j''d')$-orbit as $\varphi_{2,j}$ so in particular these two representations of the quiver $Q'$ are isomorphic. Now, we recall that the category of $\theta'$-semistable representations is abelian  and that the simple objects in this category are those $\theta'$-stable representations \cite[Lemma 10.6]{libroKirillov}. So, $\text{Hom}(\varphi_{1,j},\varphi_{2,j})\cong\text{Hom}(\varphi_1,\varphi_2)^{\oplus d_j''d_j''}$ but if $\varphi_{1,j}\cong\varphi_{2,j}$ then not all of the $d_j''\times d_j''$ summands in which the isomorphism decomposes can be zero. Schur's lemma then tells us that $\varphi_1\cong\varphi_2$ since $\varphi_1,\varphi_2$ are simple.  However, this isomorphism does not need to be unitary. But, if not, then this would contradict Theorem \ref{quiverVortexEqnsLinearCase} as every $\text{GL}(d')$-orbit intersects one and only one $\text{U}(d')$-orbit. The same argument also shows that $\psi_1\cong \psi_2$ so $F$ is indeed an injective map. 
	\\
	\\
	We summarize the previous discussion in the following theorem:
	\begin{theorem}\label{quiverSegreEmbeddingTheorem}
		Let $Q',Q''$ be quivers, $d'\in \mathbb{N}^{|Q_0'|},d''\in \mathbb{N}^{|Q_0''|}$, $\theta'\in \mathbb{R}^{|Q_0'|},\theta''\in\mathbb{R}^{|Q_0''|}$ be  generic dimension vectors and stability parameters respectively so that the corresponding symplectic reductions are manifolds. Then, the tensorization map
		{\normalfont 
			$$
			\begin{matrix}
				\mu^{-1}(\theta')/\text{U}(d')\times \mu^{-1}(\theta'')/\text{U}(d'') & \longrightarrow & \mu^{-1}(\theta)/\text{U}(d) \\ 
				([\varphi=(\varphi_\alpha)_{\alpha\in Q'_1}],[\psi=(\psi_\beta)_{\beta\in Q_1''}]) & \longmapsto & [(\varphi_\alpha\otimes\text{Id})_{(\alpha,j)\in (Q'\otimes Q'')_1}\sqcup (\text{Id}\otimes \psi_\beta)_{(i,\beta)\in (Q'\otimes Q'')_1}]
			\end{matrix}
			$$}
		
		\noindent is an embedding for $\theta=(\theta_i'+\theta_j'')_{(i,j)\in (Q'\otimes Q'')_0}$ and $d=(d_i'd_j'')_{(i,j)\in (Q'\otimes Q'')_0}$. The image is then a submanifold of real dimension $-2(\langle d',d' \rangle_{Q'}+\langle d'',d''\rangle_{Q''})$.
	\end{theorem}
	\begin{remark}
		
		The bilinear form on $\mathbb{Z}^{|Q_0|}$,
		$$
		\langle d,d'\rangle_Q:=\sum_{i\in Q_0}d_id'_i-\sum_{\alpha\in Q_1}d_{t\alpha}d'_{h\alpha}
		$$
		is the so-called \emph{Euler form} associated to $Q$ and it plays an important role in the representation theory of the quiver $Q$ (see, for instance, \cite{lectureNotesCrawley}). 
	\end{remark}
	\begin{remark}\label{algebraicQuiverSegreEmbedding}
		By Serre's GAGA \cite[Théorème 2]{GAGA} and under suitable hypotheses, for instance asking that the quiver $Q$ has no cycles so that the moduli space $\mathcal{M}^{\theta-ss}(Q,d)=\mathcal{M}^{\theta-s}(Q,d)$ is a smooth projective variety (see Section \ref{theLinearCaseSubsection}), the associated map
		$$
		\mathcal{M}^{\theta'-s}(Q',d')\times \mathcal{M}^{\theta''-s}(Q'',d'') \to \mathcal{M}^{\theta-s}(Q,d),
		$$
		to that in Theorem \ref{quiverSegreEmbeddingTheorem}, is algebraic and, in fact, a closed immersion. 
	\end{remark}
	\begin{example}\label{classicalSegreEmbedding}
		The purpose of this example is to see that the classical Segre embedding can be recovered from our previous discussions. We work now in the algebraic category. Recall the $n$-edge Kronecker quiver, $Q_n$, from Example \ref{kroneckerQuiverExample}. Our first task will be to describe the moduli space of $\theta$-semistable representations,  for $\theta=(-2,0,0,2)$, of the quiver 
		
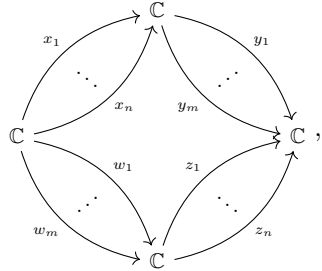
\begin{center}
			\begin{tikzcd}[scale cd=0.8, column sep =small, row sep =small]
				&&\mathbb{C}\arrow[to=3-5,bend left,"y_1"] \arrow[to=3-5,bend right,"y_m",']& &\\
				&\ddots&&\iddots& \\
				\mathbb{C} \arrow[to=1-3,bend left,"x_1"] \arrow[to=1-3,bend right,"x_n",']\arrow[to=5-3,bend left,"w_1"] \arrow[to=5-3,bend right,"w_m",']& & && \mathbb{C} \\
				&\iddots&&\ddots& \\
				&& \mathbb{C} \arrow[to=3-5,bend left,"z_1"] \arrow[to=3-5,bend right,"z_n",']&& 
			\end{tikzcd},
		\end{center}
		or in other words, following the notation from Section \ref{theLinearCaseSubsection}, the GIT quotient 
		$\mathcal{M}^{\theta-ss}(Q_n\otimes Q_m,d)$ with $d=(1,1,1,1)$. Note that $\theta$-semistability is equivalent to semistability with respect to the parameter $(-1,0,0,1)$ so for convenience we will work with the latter instead. Following Section \ref{theLinearCaseSubsection} we know that $t=(t_1,t_2,t_3,t_4)\in \text{GL}(d)$ acts on $\varphi\in V=\text{Rep}(Q_n\otimes Q_m,d)\cong \mathbb{A}^{2(n+m)}_\mathbb{C}$ by
		$$
		(t_1,t_2,t_3,t_4)\cdot \varphi=\bigg(\frac{t_2}{t_1}x_i,\frac{t_4}{t_2}y_j,\frac{t_3}{t_1}w_k,\frac{t_4}{t_3}z_l\bigg)_{\substack{i,l=1,\ldots,n\\j,k=1,\ldots,m}}
		$$
		where $\varphi=(x_i,y_j,w_k,z_l)_{\substack{i,l=1,\ldots,n\\j,k=1,\ldots,m}}$. Also we know that 
		$$
		\mathcal{M}^{\theta-ss}(Q_n\otimes Q_m,d)=\text{Proj}(\bigoplus_{\eta\geq 0}\mathcal{O}_V^{\chi_\theta^\eta})
		$$ 
		for
		\begin{align*}
			\mathcal{O}_V^{\chi_\theta^\eta}&=\{f\in \mathcal{O}_V\cong \mathbb{C}[x_i,y_j,w_k,z_l]_{\substack{i,l=1,\ldots,n\\j,k=1,\ldots,m}}| \  f(t\cdot \varphi)=\bigg(\frac{t_4}{t_1}\bigg)^\eta f(\varphi) \ \forall  \varphi\in V \text{ and }t\in \text{GL}(d)\} \\
			&=\mathbb{C}[x_iy_j,w_kz_l]_{(\eta)}.
		\end{align*}
		The subscript in the latter term of the equation means that we are taking all those homogeneous polynomials of degree $\eta$. 
		But, 
		$$
		\bigoplus_{\eta\geq 0}\mathcal{O}_V^{\chi_\theta^\eta}=\mathbb{C}[x_iy_j,w_kz_l]_{\substack{i,l=1,\ldots,n\\j,k=1,\ldots,m}}\cong \mathbb{C}[s_{ij},t_{kl}]_{\substack{i,l=1,\ldots,n\\j,k=1,\ldots,m}}/\mathcal{I},
		$$ 
		with $$\mathcal{I}=(s_{i_1j_1}s_{i_2j_2}-s_{i_1j_2}s_{i_2j_1},t_{k_1l_1}t_{k_2l_2}-t_{k_1l_2}t_{k_2l_1})_{\substack{i_1,i_2,l_1,l_2=1,\ldots,n\\j_1,j_2,k_1,k_2=1,\ldots,m}},$$ 
		via the morphism defined on generators by $s_{ij}\mapsto x_iy_j$ and $t_{kl}\mapsto w_kz_l$. 
		So, $\mathcal{M}^{\theta-ss}(Q_n\otimes Q_m,d)$ is the projective variety given by the equations defining the homogeneous ideal $\mathcal{I}$. As a matter of fact, $\mathcal{M}^{\theta-ss}(Q_n\otimes Q_m,d)$ is smooth as every $\theta$-semistable representation is in fact $\theta$-stable (see Section \ref{theLinearCaseSubsection}).
		\\
		\\
		Following the results from Section \ref{theUntwistedCase} and Remark \ref{algebraicQuiverSegreEmbedding}, the image of the morphism
		$$
		\begin{matrix}
			\mathcal{M}^{\theta'-s}(Q_n,d')&\times &\mathcal{M}^{\theta'-s}(Q_m,d')&\longrightarrow & \mathcal{M}^{\theta-s}(Q_n\otimes Q_m,d) \\
			\cong & &\cong &&
			\\\mathbb{P}^{n-1}_\mathbb{C} & &\mathbb{P}^{m-1}_\mathbb{C}&&
		\end{matrix}
		$$
		with $d'=(1,1)$ and $\theta'$ as in Example \ref{exampleProjectiveSpaceAsQuiverModuli}, induced by the tensor product of quiver representations:
		
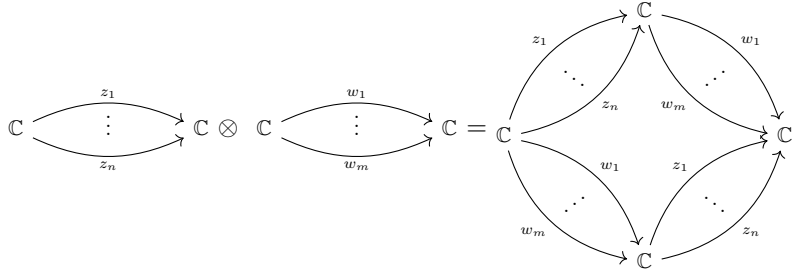
\begin{center}
			\begin{tikzcd}[scale cd=0.8]
				\mathbb{C} \arrow[to=1-3,bend left=25,"z_1"]\arrow[to=1-3,bend right=25,"z_n",']& \vdots & \mathbb{C}
			\end{tikzcd}$\otimes$ \begin{tikzcd}[scale cd=0.8]
				\mathbb{C}  \arrow[to=1-3,bend left=25,"w_1"]\arrow[to=1-3,bend right=25,"w_m",']& \vdots & \mathbb{C}
			\end{tikzcd}=\begin{tikzcd}[scale cd=0.8, column sep =small, row sep =small]
				&&\mathbb{C}\arrow[to=3-5,bend left,"w_1"] \arrow[to=3-5,bend right,"w_m",']& &\\
				&\ddots&&\iddots& \\
				\mathbb{C} \arrow[to=1-3,bend left,"z_1"] \arrow[to=1-3,bend right,"z_n",']\arrow[to=5-3,bend left,"w_1"] \arrow[to=5-3,bend right,"w_m",']& & && \mathbb{C} \\
				&\iddots&&\ddots& \\
				&& \mathbb{C} \arrow[to=3-5,bend left,"z_1"] \arrow[to=3-5,bend right,"z_n",']&& 
			\end{tikzcd},
		\end{center}
		is the subvariety of $\mathcal{M}^{\theta-s}(Q_n\otimes Q_m,d)$ cut out by the equations 
		$$
		s_{ij}-t_{ji}=0, \ i=1,\ldots,n,\ j=1,\ldots,m.
		$$
		This subvariety is precisely that which the classical Segre embedding describes.

	\end{example}
	\subsection{Character varieties of free abelian groups} \label{characterVarietiesSection} Let us consider the free abelian group $\mathbb{Z}^r$, $r\in \mathbb{N}$. The so-called $\text{GL}(n)$-\emph{character variety} of $\mathbb{Z}^r$ is the GIT quotient 
	$$
	\mathcal{M}_r(\text{GL}(n)):=\text{Hom}(\mathbb{Z}^r, \text{GL}(n))\sslash \text{GL}(n)
	$$ 
	where the action of the group $\text{GL}(n)$ on the space of representations of $\mathbb{Z}^r$, $\text{Hom}(\mathbb{Z}^r,\text{GL}(n))$, is by conjugation.  We now use tensorization of quiver reprentations to identify a distinguished closed irreducible subscheme of this character variety. We first study the case $r=2$ and later we will see that with minor modifications, our approach can be generalized for arbitrary $r\in \mathbb{N}$. 
	\\
	\\
	In the next lemma we follow the notation from Example \ref{exampleMatricesUnderConjugation} and Section \ref{moduliQuiversWithRelations}.
	\begin{lemma}\label{algebraicMorphismFromJordanQuivers}
		Let $n,m$ be positive natural numbers and let $\mathcal{R}=\{\alpha\beta-\beta\alpha\}$ be a set of relations of the quiver $Q=Q_J\otimes Q_J$ as in Figure \ref{eightQuiver}. There is a map of affine schemes
		$$
		\mathcal{M}(Q_J,n)\times \mathcal{M}(Q_J,m)\to \mathcal{M}(Q,nm,\mathcal{R})
		$$
		induced by tensorization of quiver representations.
	\end{lemma}
	\begin{proof}
		We recall from Example \ref{exampleMatricesUnderConjugation} that
		$$
		\mathcal{M}(Q_J,n)=\text{Spec}(\mathbb{C}[\alpha_1,\ldots,\alpha_{n}]^{\mathfrak{S}_{n}})
		$$
		where $\mathfrak{S}_{n}$, the symmetric group of degree $n$, acts by permuting variables and similarly for $\mathcal{M}(Q_J,m)$. On the other hand,
		$$
		\mathcal{M}(Q,nm,\mathcal{R})=\text{Spec}(\mathbb{C}[\lambda_1,\ldots,\lambda_{nm},\mu_1,\ldots,\mu_{nm}]^{\mathfrak{S}_{nm}})
		$$ 
		where $\sigma\in\mathfrak{S}_{nm}$ acts on the variables by $\lambda_{i}\mapsto \lambda_{\sigma(i)}$ and $\mu_i\mapsto \mu_{\sigma(i)}$ for all $i=1,\ldots,nm$ \cite[Theorem 4.1]{Domokos}. It is enough then to give a ring morphism
		$$
		\varphi:\mathbb{C}[\lambda_1,\ldots,\lambda_{nm},\mu_1,\ldots,\mu_{nm}]^{\mathfrak{S}_{nm}}\to \mathbb{C}[\alpha_1,\ldots,\alpha_{n}]^{\mathfrak{S}_{n}}\otimes_{\mathbb{C}}\mathbb{C}[\beta_1,\ldots,\beta_{m}]^{\mathfrak{S}_{m}}
		$$  
		in order to get a morphism of the associated affine schemes.  We claim that the map 
		$$
		p\mapsto p(\underbrace{\alpha_1,\ldots,\alpha_1}_{m\text{ times}},\ldots,\underbrace{\alpha_n,\ldots,\alpha_n}_{m\text{ times}},\underbrace{\beta_1,\ldots,\beta_m,\ldots,\beta_1,\ldots,\beta_m}_{n\text{ times}})$$
		is a well-defined morphism between these two rings since for all $p\in\mathbb{C}[\lambda_1,\ldots,\lambda_{nm},\mu_1,\ldots,\mu_{nm}]^{\mathfrak{S}_{nm}} $, $\sigma\in\mathfrak{S}_n$ and $\sigma'\in \mathfrak{S}_m$, there exist $\tau\in \mathfrak{S}_{nm}$ such that 
		$$
		\varphi(p)(\alpha_{\sigma(1)},\ldots,\alpha_{\sigma(n)},\beta_{\sigma'(1)},\ldots,\beta_{\sigma'(m)})=\varphi(\tau\cdot p)=\varphi(p).
		$$
		Indeed, given $\sigma\in \mathfrak{S}_n$, $\sigma'\in \mathfrak{S}_m$, the latter equality will hold for $\tau\in \mathfrak{S}_{nm}$ given by 
		$$
		km+j \mapsto m(\sigma(k+1)-1)+\sigma'(j), \ k=0,\ldots,n-1, \ j=1,\ldots,m. 
		$$
		It is now straightforward to see that the map induced by $\varphi$ on closed points is that coming from the tensor product 
		\vspace{-7mm}
		\begin{center}
			\begin{tikzcd}[scale cd=0.8, column sep =normal, row sep =normal]
				\mathbb{C}^n \arrow[loop,in=225,out=135,distance= 20mm,"A",']
			\end{tikzcd} $\otimes$ 		\begin{tikzcd}[scale cd=0.8, column sep =normal, row sep =normal]
				\mathbb{C}^m \arrow[loop,in=45,out=315,distance= 20mm,"B",']
			\end{tikzcd} =\begin{tikzcd}[scale cd=0.8, column sep =normal, row sep =normal]
				\mathbb{C}^n\otimes \mathbb{C}^m \arrow[loop,in=45,out=315,distance= 20mm,"\text{Id}_n\otimes B",']  \arrow[loop,in=225,out=135,distance= 20mm,"A\otimes \text{Id}_m",']
			\end{tikzcd}
		\end{center}
		\vspace{-7mm}
		of quiver representations.
	\end{proof}
	
	Now, we note that we have an open embedding 
	\begin{equation}\label{openEmbeddingIntoModuliOfJordabQuiver}
		\mathcal{M}_1(\text{GL}(n))=\text{GL}(n)\sslash\text{GL}(n)\longhookrightarrow\mathcal{M}(Q_J,n).
	\end{equation}
	As a matter of fact, the left-most GIT quotient is a distinguished open subset of the affine scheme $\mathcal{M}(Q_J,n)=\text{Spec}(\mathbb{C}[\alpha_1,\ldots,\alpha_n]^{\mathfrak{S}_n})$. To see this we note in the first place that 
	$$
	(\mathbb{C}[\alpha_1,\ldots,\alpha_n]^{\mathfrak{S}_n})_{(\alpha_1\ldots\alpha_n)}\cong(\mathbb{C}[\alpha_1,\ldots,\alpha_n]_{(\alpha_1\ldots\alpha_n)})^{\mathfrak{S}_n}.
	$$
	This comes from the fact that taking invariants under the action of a finite group commutes with taking localizations \cite[Chapter 5, Exercise 12]{AtiyahMacdonald}. In second place, we have that 
	\begin{equation}\label{chevalleyCharacter}
		\text{Spec}((\mathbb{C}[\alpha_1,\ldots,\alpha_n]_{(\alpha_1\ldots\alpha_n)})^{\mathfrak{S}_n})= T_n\sslash \mathfrak{S}_n\cong \text{GL}(n)\sslash\text{GL}(n)
	\end{equation}
	where $T_n=\overbrace{\mathbb{G}_m\times\cdots\times \mathbb{G}_m}^{n \text{ times}}$. In fact, the isomorphism in Equation (\ref{chevalleyCharacter}) is a particular instance of a more general fact and this is that $$\mathcal{M}_r(\text{GL}(k))\cong T_{k}^r\sslash \mathfrak{S}_{k}$$ 
	for all $r,k\in \mathbb{N}$ (see, for instance, \cite[Proposition 5.10]{FlorentinoSilva}). 
	\\
	\\
	The next step we take is to note that the open set $\mathcal{M}_1(\text{GL}(n))\times\mathcal{M}_1(\text{GL}(m))$ of the fibered product $\mathcal{M}(Q_J,n)\times \mathcal{M}(Q_J,m)$ is irreducible. This follows from both the fact that open subsets of irreducible topological spaces are irreducible ($\text{GL}(n)\sslash\text{GL}(n)$ is an open subset of $\mathcal{M}(Q_J,n)\cong \mathbb{A}^n_\mathbb{C}$ which is irreducible) and the fact that fibered products of irreducible varieties are irreducible. The latter is essentially because tensor products of $\mathbb{C}$-algebras without zero divisors, has no zero divisors. For a full proof of the latter we refer the reader to Iitaka \cite[p.97]{IitakaAG}.
	\\
	\\
	The same argument as that used to show the open embedding in Equation \ref{openEmbeddingIntoModuliOfJordabQuiver} gives that there is an open embedding 
	$$
	\mathcal{M}_2(\text{GL}(mn))\cong T_{nm}^2\sslash \mathfrak{S}_{nm} \longhookrightarrow \mathcal{M}(Q,nm,\mathcal{R})
	$$
	and, moreover, that the left-most quotient is also a distinguished open subset of $\mathcal{M}(Q,mn,\mathcal{R})$.
	Putting all together, we conclude that the restriction of the morphism in Lemma \ref{algebraicMorphismFromJordanQuivers} gives a morphism
	\begin{equation}\label{morphismTensorizationCharacterVarieties}
		\mathcal{M}_1(\text{GL}(n))\times \mathcal{M}_1(\text{GL}(m))\longrightarrow \mathcal{M}_2(\text{GL}(mn))
	\end{equation}
	whose set-theoretic image is irreducible.  Furthermore, this morphism is quasi-compact, since it is a morphism of affine schemes. Therefore, the scheme-theoretical image will be the affine closed subscheme of the character variety $\mathcal{M}_2(\text{GL}(mn))$ determined by the kernel of the ring map
	$$
	\mathbb{C}[\lambda_1,\ldots,\lambda_{nm},\mu_1,\ldots,\mu_{nm}]^{\mathfrak{S}_{nm}}_{(\lambda_1\cdots\lambda_{mn}\mu_1\cdots\mu_{mn})}\to (\mathbb{C}[\alpha_1,\ldots,\alpha_{n}]^{\mathfrak{S}_{n}}\otimes_{\mathbb{C}}\mathbb{C}[\beta_1,\ldots,\beta_{m}]^{\mathfrak{S}_{m}})_{(\alpha_1\cdots\alpha_n\beta_1\cdots\beta_m)},
	$$
	induced from the ring morphism $\varphi$ in Lemma  \ref{algebraicMorphismFromJordanQuivers}, and whose underlying topological space is the closure of the set-theoretical image (see \cite[Theorem 3.42 and Proposition 10.30]{GoertzAlgebraicGeometry}). We refer to the closed scheme we have just described as the \emph{Segre subscheme} of the character variety $\mathcal{M}_2(\text{GL}(mn))$ since it resembles the stream of ideas we have been following since Section \ref{segreEmbeddingForQuiverRepresentations}.
	\\
	\\
	As hinted at the beginning of this section, with minor modifications, the same technique can be used to obtain Segre subschemes of the character variety $\mathcal{M}_r(\text{GL}(n))$ for $r>2$. 
	Our first remark on this matter is that, as in Lemma \ref{algebraicMorphismFromJordanQuivers}, there is a morphism of affine schemes
	\begin{equation}\label{generalizedAlgebraicMorphismFromJordanQuivers}
		\mathcal{M}(Q_J,n_1)\times \cdots\times\mathcal{M}(Q_J,n_r)\longrightarrow \mathcal{M}(Q,n_1,\ldots n_r,\mathcal{R})
	\end{equation}
	coming from tensorization of quiver representations. Here $Q=Q_J^{\otimes r}$ (see Figure \ref{rLacedQuiver}) and $\mathcal{R}=\{\alpha_i\alpha_j-\alpha_j\alpha_i|i,j=1,\ldots,r\}$ so that  $\mathcal{M}(Q,n_1\ldots n_r,\mathcal{R})$ is obtained by taking the GIT quotient of the variety of $r$ commuting square matrices by the action of $\text{GL}(n_1\cdots n_r)$ via simultaneous conjugation. As in Lemma \ref{algebraicMorphismFromJordanQuivers}, the strategy to show that the claimed morphism in Equation (\ref{generalizedAlgebraicMorphismFromJordanQuivers}) exists is to give a morphism between the coordinate rings of the  corresponding affine schemes. The combinatorial argument to see the well-definedness of this ring morphism generalizes without issue. Secondly, as in Equation (\ref{morphismTensorizationCharacterVarieties}), the map from Equation (\ref{generalizedAlgebraicMorphismFromJordanQuivers}) restricts to a morphism
	$$\mathcal{M}_1(\text{GL}(n_1))\times \cdots\times \mathcal{M}_1(\text{GL}(n_r))\to\mathcal{M}_r(\text{GL}(n_1\cdots n_r)).$$
	The Segre suscheme of $\mathcal{M}_r(\text{GL}(n_1\cdots n_r)$ is then the scheme theoretic image under this morphism. 
	\vspace{-7mm}
	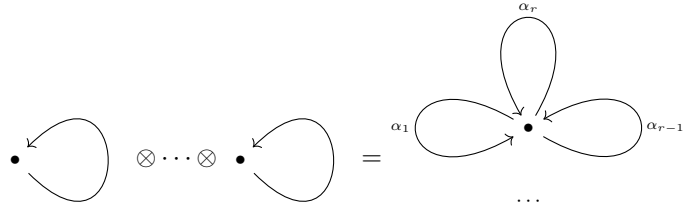
\begin{figure}[H]
		\begin{center}
			\begin{tikzcd}[scale cd=0.8,column sep =normal, row sep =normal]
				\bullet \arrow[loop,in=45,out=315,distance= 20mm]
			\end{tikzcd}$\otimes \cdots\otimes$ \begin{tikzcd}[scale cd=0.8, column sep =normal, row sep =normal]
				\bullet \arrow[loop,in=45,out=315,distance= 20mm]
			\end{tikzcd}=\begin{tikzcd}[scale cd=0.8, column sep =normal, row sep =normal]
				\bullet \arrow[loop,out=60,in=120,distance=20mm,"\alpha_r",']\arrow[loop,in=30,out=330,distance= 20mm,"\alpha_{r-1}",']  \arrow[loop,in=210,out=150,distance= 20mm,"\alpha_1",'] \\ \cdots
			\end{tikzcd}
		\end{center}
		\vspace{-7mm}
		\caption{\label{rLacedQuiver} Tensor product of $r$ Jordan quivers}
	\end{figure}	
	\nocite{*}
	\bibliography{biblio}
\end{document}